\documentclass[12pt,reqno]{amsart}
\usepackage{amsmath,amsfonts,amsthm,amsopn,amssymb,mathrsfs,enumerate,color}
\usepackage{amsmath}
\usepackage{comment}

\usepackage{cite,marginnote}
\usepackage{mathtools} 

\usepackage{color,graphicx,enumerate}
\usepackage[colorlinks=true,urlcolor=blue,
citecolor=red,linkcolor=blue,linktocpage,pdfpagelabels,
bookmarksnumbered,bookmarksopen]{hyperref}
\usepackage[english]{babel}

\usepackage[left=2.9cm,right=2.9cm,top=2.8cm,bottom=2.8cm]{geometry}

\numberwithin{equation}{section}

\makeindex
\newtheorem{theorem}{Theorem}[section]

\newtheorem{corollary}{Corollary}[section]
\newtheorem{remark}{Remark}[section]

\newtheorem{definition}{Definition}[section]
\newtheorem{proposition}{Proposition}[section]
\newtheorem{lemma}{Lemma}[section]

\newtheorem{iteration lemma}{iteration Lemma}[section]

\newcommand{\bt}{\begin{theorem}}
\newcommand{\et}{\end{theorem}}
\newcommand{\bl}{\begin{lemma}}
\newcommand{\el}{\end{lemma}}
\newcommand{\bd}{\begin{definition}}
\newcommand{\ed}{\end{definition}}
\newcommand{\bc}{\begin{corollary}}
\newcommand{\ec}{\end{corollary}}
\newcommand{\bp}{\begin{proof}}
\newcommand{\ep}{\end{proof}}
\newcommand{\bx}{\begin{example}}
\newcommand{\ex}{\end{example}}
\newcommand{\bi}{\begin{exercise}}
\newcommand{\ei}{\end{exercise}}
\newcommand{\bo}{\begin{proposition}}
\newcommand{\eo}{\end{proposition}}
\newcommand{\br}{\begin{remark}}
\newcommand{\er}{\end{remark}}
\newcommand{\beq}{\begin{equation}}
\newcommand{\eeq}{\end{equation}}
\newcommand{\ba}{\begin{align}}
\newcommand{\ea}{\end{align}}
\newcommand{\bn}{\begin{enumerate}}
\newcommand{\en}{\end{enumerate}}
\newcommand{\bg}{\begin{align*}}
\newcommand{\bcs}{\begin{cases}}
\newcommand{\ecs}{\end{cases}}

\newcommand{\bean}{\begin{eqnarray*}}
\newcommand{\eean}{\end{eqnarray*}}

\def\bd{\mathrm{bd}\,}

\usepackage{amsmath,amssymb,amsthm,mathtools}
\usepackage{hyperref}

\theoremstyle{remark}

\title[Stability for Affine Fractional Sobolev]{Sharp Stability for the Affine Fractional Sobolev Inequality}

\author[S.~Fan]{Song Fan}
\author[G-D.~Li]{Gui-Dong Li}
\author[J.~J.~Zhang]{Jianjun Zhang}

\address[S.~Fan]{\newline\indent School  of  Mathematics  and  Statistics
\newline\indent
Guizhou University
\newline\indent
Guiyang, 550025, Guizhou, PR China}
\email{\href{mailto:doraemonsong77@gmail.com}{doraemonsong77@gmail.com}}

\address[G-D.~Li]{\newline\indent School  of  Mathematics  and  Statistics
\newline\indent
Guizhou University
\newline\indent
Guiyang, 550025, Guizhou, PR China}
\email{\href{mailto:bestdong123@163.com}{bestdong123@163.com}}

\address[J.~J.~Zhang]{\newline\indent College of Mathematics and Statistics
\newline\indent
Chongqing Jiaotong University
\newline\indent
Xuefu, Nan'an, 400074, Chongqing, PR China}
\email{\href{mailto:zhangjianjun09@tsinghua.org.cn}{zhangjianjun09@tsinghua.org.cn}}

\begin{document}

\begin{abstract}
We prove a sharp quantitative stability result for the affine fractional
\(L^2\)-Sobolev inequality in \(\dot H^s(\mathbb R^n)\), \(0<s<1\), introduced by
Haddad--Ludwig (\emph{Math. Ann.} \textbf{388} (2024), 1091--1115). In
particular, we identify the kernel of the affine Hessian, determine the sharp
local spectral gap, and show that the optimal global stability constant is
strictly smaller than the corresponding local spectral value.

\vskip0.23in

\noindent{\it   {\bf Key  words:} Sharp quantitative stability, Affine fractional Sobolev inequality, fractional Sobolev inequality.}


\end{abstract}

\maketitle

\section{Introduction}
Sharp Sobolev inequalities, together with their quantitative stability
counterparts, occupy a central position in geometric analysis and nonlinear
PDE. In the Euclidean case, the sharp Sobolev inequality and the complete
classification of its extremals go back to the  works of Aubin and
Talenti~\cite{Aubin76,Talenti76}. The corresponding fractional theory follows
from Lieb's sharp Hardy--Littlewood--Sobolev inequality~\cite{Lieb83}. Let
\(0<s<\frac n2\),  and set \(2_s^*=\frac{2n}{n-2s}\).
With the normalization of the homogeneous fractional seminorm given by
\begin{equation}\label{eq:seminorm}
[u]_{\dot H^s(\mathbb R^n)}^2
:=
\int_{\mathbb R^n}
|\zeta|^{2s}|\widehat u(\zeta)|^2\,d\zeta
\qquad \text{where}\quad
\widehat u(\zeta)
=
\int_{\mathbb R^n}e^{-ix\cdot\zeta}u(x)\,dx,
\end{equation}
the sharp fractional Sobolev inequality takes the form
\begin{equation}
\label{eq:intro-frac-sob}
[u]_{\dot H^s(\mathbb R^n)}^2
\ge
S_{n,s}\|u\|_{L^{2_s^*}(\mathbb R^n)}^2 .
\end{equation}
Moreover, equality in \eqref{eq:intro-frac-sob} holds if and only if \(u\)
belongs to the conformal family \cite{Lieb83}
\[
U_{c,\lambda,x_0}(x)
=
c\left(\frac{\lambda}{\lambda^2+|x-x_0|^2}\right)^{\frac{n-2s}{2}},
\qquad
c\in\mathbb R,\quad \lambda>0,\quad x_0\in\mathbb R^n .
\]
For background on the fractional Laplacian, the Gagliardo seminorm
\eqref{eq:seminorm}, and the Fourier characterization of fractional Sobolev
spaces, we refer to Di Nezza--Palatucci--Valdinoci~\cite{DPV12}.

The quantitative stability problem for \eqref{eq:intro-frac-sob}, in the sense
of Bianchi--Egnell \cite{BE91}, was solved in the fractional setting by
Chen--Frank--Weth \cite{CFW13}. More precisely, if
\begin{equation}
\label{eq:intro-classical-manifold}
\mathcal M
:=
\left\{
U_{c,\lambda,x_0}:
c\in\mathbb R,\ \lambda>0,\ x_0\in\mathbb R^n
\right\},
\end{equation}
then there exists \(\alpha_{n,s}>0\) such that
\begin{equation}
\label{eq:intro-CFW-stability}
d_{\operatorname{frac}}(u,\mathcal M)^2
\ge
[u]_{\dot H^s(\mathbb R^n)}^2
-
S_{n,s}\|u\|_{L^{2_s^*}(\mathbb R^n)}^2
\ge
\alpha_{n,s}
d_{\operatorname{frac}}(u,\mathcal M)^2,
\end{equation}
where
\[
d_{\operatorname{frac}}(u,\mathcal M)^2
:=
\inf_{V\in\mathcal M}
[u-V]_{\dot H^s(\mathbb R^n)}^2.
\]

Affine Sobolev theory provides a genuinely stronger and more geometric
refinement of the classical Sobolev inequality. The endpoint affine
Sobolev--Zhang inequality was established by Zhang~\cite{Zhang99}, and the
sharp affine \(L^p\)-Sobolev inequality was proved by
Lutwak--Yang--Zhang~\cite{LYZ02}, building on the \(L_p\) affine
isoperimetric theory developed in~\cite{LYZ00}. Further developments include
asymmetric affine Sobolev inequalities~\cite{HS09}, affine symmetrization and
affine logarithmic Sobolev inequalities~\cite{HSX12}, affine
Moser--Trudinger and Morrey--Sobolev inequalities~\cite{CLYZ09}, affine
\(BV\)-Sobolev inequalities~\cite{Wang12}, and affine Sobolev-type
inequalities obtained through the \(L_p\) Busemann--Petty centroid
inequality~\cite{HJM16,HJM19}. Fractional and trace variants have also been
studied in \cite{DHJM18,HL24,HL25,HL25-1,CPS15,Bullion26}.

Haddad--Ludwig~\cite{HL24} proved the sharp affine fractional
\(L^p\)-Sobolev inequality. In the Hilbertian case \(p=2\), the corresponding
directional energy can be written on the Fourier side as
\begin{equation}
\label{eq:intro-Axi-frac}
A_\xi(u)
:=
\int_{\mathbb R^n}
|\xi\cdot\zeta|^{2s}|\widehat u(\zeta)|^2\,d\zeta,
\qquad \xi\in\mathbb S^{n-1}.
\end{equation}
With
\( \left\langle g\right\rangle_{\mathbb S^{n-1}}
=
\frac1{|\mathbb S^{n-1}|}
\int_{\mathbb S^{n-1}}g(\xi)\,d\sigma(\xi)\),
rotation invariance gives
\begin{equation}
\label{eq:intro-average-A}
\left\langle A_\xi(u)\right\rangle_{\mathbb S^{n-1}}
=
m_{n,s}[u]_{\dot H^s(\mathbb R^n)}^2, \qquad  m_{n,s}
=
\left\langle|\omega_1|^{2s}\right\rangle_{\omega\in\mathbb S^{n-1}} .
\end{equation}
We then define the normalized affine fractional energy by
\begin{equation}
\label{eq:intro-affine-energy}
\mathcal E_{\operatorname{aff},s}(u)^2
:=
m_{n,s}^{-1}
\left\langle
A_\xi(u)^{-r}
\right\rangle_{\mathbb S^{n-1}}^{-1/r}, \qquad
r=\frac n{2s}.
\end{equation}
With this convention, the sharp affine fractional Sobolev inequality of
Haddad--Ludwig~\cite{HL24}  takes the
form
\begin{equation}
\label{eq:intro-aff-frac-sob-p=2}
[u]_{\dot H^s(\mathbb R^n)}^2
\ge
\mathcal E_{\operatorname{aff},s}(u)^2
\ge
S_{n,s}\|u\|_{L^{2_s^*}(\mathbb R^n)}^2.
\end{equation}
Thus the affine fractional Sobolev inequality is a genuine strengthening of
the classical sharp fractional Sobolev inequality \eqref{eq:intro-frac-sob}.

The quantitative stability problem for sharp Sobolev-type inequalities goes
back to Br\'ezis--Lieb~\cite{BL85} and Bianchi--Egnell~\cite{BE91}. In the
fractional Sobolev setting, the global stability theorem was proved by
Chen--Frank--Weth~\cite{CFW13}; the sharp stability constant and related
asymptotic problems were further studied in
\cite{Konig23,Konig25,CLT25}. Recent developments in stability theory include
\cite{CLT23,CGT25,DEFFL25,FZ22,ZZ25,ZZZ25,FLZ26,Neumayer26}, among others. In the affine
setting, stability is less developed; we mention in particular Wang's
quantitative result for affine symmetrization~\cite{Wang13}, Nguyen's
stability theorem for the affine \(BV\)-Sobolev inequality~\cite{Nguyen16},
and the anisotropic \(BV\)-stability theory of
Figalli--Maggi--Pratelli~\cite{FMP13}. More recently, Fan--Li--Zhang~\cite{FLZ26B}
established sharp quantitative stability for the affine \(p\)-Sobolev inequality.
The present work concerns the sharp
global stability of \eqref{eq:intro-aff-frac-sob-p=2}.

Let \(U(x)=\left(1+|x|^2\right)^{-\frac{n-2s}{2}}\) .
The equality cases in \eqref{eq:intro-aff-frac-sob-p=2} are the affine images
of the classical fractional Sobolev extremals \cite{HL24}. Thus the natural extremal cone is
\[
\mathcal M_{\rm cone}
:=
\left\{
c\,T_{A,x_0}^{-1}U:
c\in\mathbb R,\ A\in GL^+(n),\ x_0\in\mathbb R^n
\right\},
\]
where
\[
(T_{A,x_0}u)(x)
:=
|\det A|^{\frac{n-2s}{2n}}u(Ax+x_0).
\]
Correspondingly, we define
\[
D_{\operatorname{aff}}(u,\mathcal M_{\rm cone})^2
:=
\inf_{c\in\mathbb R}
\inf_{A\in GL^+(n),\,x_0\in\mathbb R^n}
\left[
T_{A,x_0}u-cU
\right]_{\dot H^s(\mathbb R^n)}^2 .
\]

Our main theorem is the following global sharp stability estimate.

\begin{theorem}
\label{thm:two-sided-affine-stability}
Assume \(0<s<1\) and \(n>2s\). Then there exists \(c_{n,s}>0\) such that, for every
\(u\in\dot H^s(\mathbb R^n)\),
\begin{equation}
\label{eq:intro-global-stability}
c_{n,s}D_{\operatorname{aff}}(u,\mathcal M_{\rm cone})^2
\le
\mathcal E_{\operatorname{aff},s}(u)^2
-
S_{n,s}\|u\|_{L^{2_s^*}(\mathbb R^n)}^2
\le
D_{\operatorname{aff}}(u,\mathcal M_{\rm cone})^2 .
\end{equation}
Moreover, the exponent \(2\) is sharp.
\end{theorem}

The affine fractional stability \eqref{eq:intro-global-stability} is not a formal consequence of the
classical fractional stability theorem of Chen--Frank--Weth~\cite{CFW13}. This
distinction is already visible at the level of the energy. In the local
Hilbertian case \(s=1\), the directional energies satisfy
\[
A_\xi(u)
=
\int_{\mathbb R^n}|\xi\cdot\nabla u|^2\,dx
=
\xi\cdot
\left(
\int_{\mathbb R^n}\nabla u\otimes\nabla u\,dx
\right)\xi ,
\]
so the affine energy is governed by a finite-dimensional positive matrix. After
an affine normalization, the problem can therefore be reduced to the classical
Sobolev stability of Bianchi--Egnell~\cite{BE91}. For \(0<s<1\), no such
finite-dimensional reduction is available: the quantities \(A_\xi(u)\) are
genuinely nonlocal Fourier-side quadratic forms, and the affine correction must
be controlled at the level of the full family \(\{A_\xi\}_{\xi\in\mathbb S^{n-1}}\).

The proof has two main parts. First, we compute the Hessian of the affine
deficit at the radial bubble \(U\). The affine Hessian decomposes as
\(\mathcal Q_U
=
\mathcal Q_{\operatorname{frac},s}
-
\mathcal R_s\),
where \(\mathcal Q_{\operatorname{frac},s}\) is the classical fractional
Sobolev Hessian \eqref{eq:def-hessian-frac} and \(\mathcal R_s\) is a nonnegative variance correction \eqref{eq:def-Rs}
coming from the directional affine energy. This correction enlarges the
classical conformal kernel by the trace-free affine Jacobi fields
\(x\cdot B\nabla U\), \( B=B^T\), \(\operatorname{tr}B=0\).
Combining the conformal spectral decomposition of the classical fractional
Hessian with a Funk--Hecke analysis of \(\mathcal R_s\), we identify
\[
\ker \mathcal Q_U
=
T_U\mathcal M_{\rm cone}
=
\operatorname{span}
\left\{
U,\ \partial_{x_1}U,\dots,\partial_{x_n}U,\
\frac{n-2s}{2}U+x\cdot\nabla U,\
x\cdot B\nabla U
\right\}.
\]

Second, we pass from local coercivity to the global estimate. This requires
controlling the affine invariances and excluding loss of compactness under
translations and dilations. We combine affine normalization, fractional profile
decomposition, and a no-splitting argument for near-extremizing sequences. The
local spectral estimate then yields the global stability inequality
\eqref{eq:intro-global-stability}.

We also consider the optimal global quotient
\[
C_{\operatorname{aff},s}
:=
\inf_{u\notin\mathcal M_{\rm cone}}
\frac{
\mathcal E_{\operatorname{aff},s}(u)^2
-
S_{n,s}\|u\|_{L^q(\mathbb R^n)}^2
}{
D_{\operatorname{aff}}(u,\mathcal M_{\rm cone})^2
}.
\]
In analogy with the classical Bianchi--Egnell quotient studied by
K\"onig~\cite{Konig23,Konig25}, we prove the following strict upper bound.

\begin{theorem}
\label{thm:affine-strict-upper-bound}
Assume \(0<s<1\) and \(n>2s\). Then
\[
C_{\operatorname{aff},s}
<
\frac{2s}{\frac n2+s+1}.
\]
\end{theorem}

The paper is organized as follows. In Section~\ref{sec:variation} we  compute the second variation of
the affine deficit at the radial bubble \(U\). Section~\ref{sec:global}
establishes the compactness theorem for normalized affine near-extremizers.
Section~\ref{sec:proof} proves Theorem~\ref{thm:two-sided-affine-stability}
and Theorem~\ref{thm:affine-strict-upper-bound}. Finally,
Section~\ref{sec:sector-decomposition-affine-frac} carries out the sectorial
spectral analysis of the affine Hessian and identifies its kernel.

\section{Variational structure of the affine fractional Sobolev deficit}

\label{sec:variation}

\subsection{Second variation at the radial bubble}

Let \(\mathfrak a_\xi\) be the bilinear form associated with \(A_\xi\):
\[
\mathfrak a_\xi(u,v)
:=
\int_{\mathbb R^n}
|\xi\cdot\zeta|^{2s}
\widehat u(\zeta)\overline{\widehat v(\zeta)}\,d\zeta .
\]
Thus \(A_\xi(u)=\mathfrak a_\xi(u,u)\).
Since \(U\) is radial, \(\widehat U(\zeta)=\widehat U_0(|\zeta|)\). Hence, by
polar coordinates \(\zeta=\rho\theta\),
\begin{equation}
\label{eq:Axu-1}
\begin{aligned}
A_\xi(U)
&=
\int_0^\infty
\rho^{n-1+2s}|\widehat U_0(\rho)|^2
\left(
\int_{\mathbb S^{n-1}}|\xi\cdot\theta|^{2s}\,d\sigma(\theta)
\right)d\rho \\
&  =
|\mathbb S^{n-1}|m_{n,s}
\int_0^\infty
\rho^{n-1+2s}|\widehat U_0(\rho)|^2\,d\rho
=:A_0 .
\end{aligned}
\end{equation}

For \(\phi\in\dot H^s(\mathbb R^n)\), set
\begin{equation}
\label{eq:def-Lxi}
L_\xi(\phi)
:=
\mathfrak a_\xi(U,\phi)
=
\int_{\mathbb R^n}
|\xi\cdot\zeta|^{2s}
\widehat U(\zeta)\overline{\widehat\phi(\zeta)}\,d\zeta .
\end{equation}
Then
\[
A_\xi(U+t\phi)
=
A_0+2tL_\xi(\phi)+t^2A_\xi(\phi).
\]

We define the affine Hessian \(\mathcal Q_U\) by
\[
\delta_{\operatorname{aff},s}(U+t\phi)
=
t^2\mathcal Q_U(\phi)
+
o(t^2),
\]
where
\[
\delta_{\operatorname{aff},s}(u)
:=
\mathcal E_{\operatorname{aff},s}(u)^2
-
S_{n,s}\|u\|_{L^{2_s^*}(\mathbb R^n)}^2 .
\]

For a positive function \(a:\mathbb S^{n-1}\to(0,\infty)\), define
\[
\Phi(a)
:=
\left\langle a(\xi)^{-r}\right\rangle_{\mathbb S^{n-1}}^{-1/r},
\qquad r=\frac n{2s}.
\]
A direct expansion of \(\Phi\) at the constant function \(A_0\) gives
\begin{equation}
\label{eq:affine-energy-second-expansion}
\begin{aligned}
\mathcal E_{\operatorname{aff},s}(U+t\phi)^2
&=
m_{n,s}^{-1}A_0
+
2m_{n,s}^{-1}t
\left\langle L_\xi(\phi)\right\rangle_{\mathbb S^{n-1}}
\\
&\quad
+
t^2
\left(
[\phi]_{\dot H^s(\mathbb R^n)}^2
-
m_{n,s}^{-1}
\frac{n+2s}{sA_0}
\operatorname{Var}_{\xi\in\mathbb S^{n-1}}
\bigl(L_\xi(\phi)\bigr)
\right)
+
o(t^2),
\end{aligned}
\end{equation}
where
\begin{equation}
\label{eq:def-var}
\operatorname{Var}_{\xi}(L_\xi)
:=
\left\langle L_\xi^2\right\rangle_{\mathbb S^{n-1}}
-
\left\langle L_\xi\right\rangle_{\mathbb S^{n-1}}^2 .
\end{equation}

We now expand the \(L^{2_s^*}\)-term. As \(t\to0\),
\[
\begin{aligned}
\|U+t\phi\|_{L^{2_s^*}(\mathbb R^n)}^2
&=
\|U\|_{L^{2_s^*}(\mathbb R^n)}^2
+
2t
\|U\|_{L^{2_s^*}(\mathbb R^n)}^{2-2_s^*}
\int_{\mathbb R^n}U^{2_s^*-1}\phi\,dx
\\
&\quad
+
t^2
\left[
(2_s^*-1)
\|U\|_{L^{2_s^*}(\mathbb R^n)}^{2-2_s^*}
\int_{\mathbb R^n}U^{2_s^*-2}\phi^2\,dx
\right.
\\
&\qquad\qquad\left.
+
(2-2_s^*)
\|U\|_{L^{2_s^*}(\mathbb R^n)}^{2-2\cdot 2_s^*}
\left(
\int_{\mathbb R^n}U^{2_s^*-1}\phi\,dx
\right)^2
\right]
+
o(t^2).
\end{aligned}
\]
Combining this expansion with \eqref{eq:affine-energy-second-expansion}, and
using the extremality of \(U\), we obtain
\[
\mathcal Q_U(\phi)
=
\mathcal Q_{\operatorname{frac},s}(\phi)
-
\mathcal R_s(\phi),
\]
where
\begin{equation}
\label{eq:def-Rs}
\mathcal R_s(\phi)
:=
m_{n,s}^{-1}
\frac{n+2s}{sA_0}
\operatorname{Var}_{\xi\in\mathbb S^{n-1}}
\bigl(L_\xi(\phi)\bigr),
\end{equation}
and
\begin{equation}
\label{eq:def-hessian-frac}
\begin{aligned}
\mathcal Q_{\operatorname{frac},s}(\phi)
&:=
[\phi]_{\dot H^s(\mathbb R^n)}^2
-
S_{n,s}(2_s^*-1)
\|U\|_{L^{2_s^*}(\mathbb R^n)}^{2-2_s^*}
\int_{\mathbb R^n}U^{2_s^*-2}\phi^2\,dx
\\
&\quad
+
S_{n,s}(2_s^*-2)
\|U\|_{L^{2_s^*}(\mathbb R^n)}^{2-2\cdot 2_s^*}
\left(
\int_{\mathbb R^n}U^{2_s^*-1}\phi\,dx
\right)^2 .
\end{aligned}
\end{equation}

\subsection{Sharp local spectral gap at the bubble}

We begin with the linearized problem at the radial extremal \(U\). By the
kernel identification in Section~\ref{sec:sector-decomposition-affine-frac},
\[
\mathcal K_U
:=
T_U\mathcal M_{\rm cone}
=
\ker\mathcal Q_U
=
\operatorname{span}
\left\{
U,\ Z_0,\ \partial_{x_1}U,\dots,\partial_{x_n}U,\
x\cdot B\nabla U
\right\},
\]
where
\[
Z_0=\frac{n-2s}{2}U+x\cdot\nabla U,
\qquad
B=B^T,\quad \operatorname{tr}B=0 .
\]

Define
\[
\Gamma_s
:=
\inf_{\substack{
\phi\perp_{\dot H^s(\mathbb R^n)}\mathcal K_U\\
\phi\neq0
}}
\frac{\mathcal Q_U(\phi)}
{[\phi]_{\dot H^s(\mathbb R^n)}^2}.
\]
The following Proposition is inspirit by \cite[Proposition~2]{Konig23}.
\begin{proposition}
\label{pro:sharp-local-affine-coeff}
Assume \(n\ge2\) and \(0<s<1\). Then
\(\Gamma_s= \gamma_s:=\frac{2s}{\frac n2+s+1}\).
Equivalently,
\begin{equation}
\label{eq:sharp-affine-gap}
\mathcal Q_U(\phi)
\ge
\gamma_s
[\phi]_{\dot H^s(\mathbb R^n)}^2,
\qquad
\phi\perp_{\dot H^s(\mathbb R^n)} \mathcal K_U.
\end{equation}
Moreover, equality in \eqref{eq:sharp-affine-gap} holds if and only if
\(\phi=Pv_2\),
where \(P\) is the stereographic conformal transfer \eqref{eq:def-conformal-transfor}, and \(v_2\in\mathcal H_2(\mathbb S^n)\) has the form
\[
v_2(\Theta)
=
a\left(\Theta_{n+1}^2-\frac1{n+1}\right)
+
\sum_{i=1}^n b_i\,\Theta_i\Theta_{n+1},
\qquad
a,b_1,\dots,b_n\in\mathbb R .
\]
\end{proposition}

\begin{proof}
Assume first that
\[
        \sum_{\ell=0}^{\infty}
\sum_{m=1}^{d_\ell}
f_{\ell,m}(r)Y_{\ell,m}(\theta)= \phi\perp_{\dot H^s(\mathbb R^n)}\mathcal K_U .
\]
In the radial sector \(\ell=0\), one has \(\mathcal R_s[\phi_0]=0\). Since the
orthogonality to \(\mathcal K_U\) removes the radial kernel directions \(U\) and
\(Z_0\), the first remaining classical quotient is \(\gamma_s\). Hence
\[
\mathcal Q_U[\phi_0]
\ge
\gamma_s[\phi_0]_{\dot H^s(\mathbb R^n)}^2.
\]

In the first angular sector \(\ell=1\), by \eqref{eq:appendix-rho-odd-zero},
again \(\mathcal R_s[\phi_1]=0\). Orthogonality to \(\mathcal K_U\) removes the
translation modes, and therefore
\[
\mathcal Q_U[\phi_1]
\ge
\gamma_s[\phi_1]_{\dot H^s(\mathbb R^n)}^2.
\]

In the degree-two sector, Corollary~\ref{cor:sharp-degree-two-sector-gap} yields
\[
\mathcal Q_U[\phi_2]
\ge
\gamma_s[\phi_2]_{\dot H^s(\mathbb R^n)}^2,
\]
because the condition
\(\phi\perp_{\dot H^s(\mathbb R^n)}\mathcal K_U\) removes all trace-free affine
zero modes \(x\cdot B\nabla U\).

For odd sectors \(\ell\ge3\), the affine correction vanishes. Thus
\[
\mathcal Q_U[\phi_\ell]
=
\mathcal Q_{\operatorname{frac},s}[\phi_\ell]
\ge
\left(1-\frac{\Lambda_1}{\Lambda_\ell}\right)
[\phi_\ell]_{\dot H^s(\mathbb R^n)}^2
\ge
\gamma_s[\phi_\ell]_{\dot H^s(\mathbb R^n)}^2.
\]

For even sectors \(\ell\ge4\), Lemma~\ref{lem:appendix-even-ge-four} gives
\[
\mathcal Q_U[\phi_\ell]
\ge
\gamma_s[\phi_\ell]_{\dot H^s(\mathbb R^n)}^2.
\]

Summing over all sectors, we obtain
\[
\mathcal Q_U(\phi)
=
\sum_{\ell=0}^{\infty}\mathcal Q_U[\phi_\ell]
\ge
\gamma_s
\sum_{\ell=0}^{\infty}
[\phi_\ell]_{\dot H^s(\mathbb R^n)}^2
=
\gamma_s[\phi]_{\dot H^s(\mathbb R^n)}^2
\]
for every
\(\phi\perp_{\dot H^s(\mathbb R^n)}T_U\mathcal M_{\rm cone}\).
Therefore \(\Gamma_s\ge \gamma_s\).

We next identify the equality case. Equality in \eqref{eq:sharp-affine-gap}
can only occur in sectors whose sharp quotient is \(\gamma_s\). By the
conformal spectral decomposition, this means that only conformal spherical
degree \(2\) can contribute.

We use the standard identification of spherical harmonics with restrictions of
homogeneous harmonic polynomials; see \cite[Section~1.1]{DX13}. In degree two,
\[
\mathcal H_2(\mathbb S^n)
=
\left\{
\Theta\mapsto \Theta\cdot C\Theta:
C=C^T,\ \operatorname{tr}C=0
\right\}.
\]
Write
\( \Theta=(\Theta',\Theta_{n+1})\), \(\Theta'=(\Theta_1,\dots,\Theta_n)\).
If
\[
C=
\begin{pmatrix}
D & b\\
b^T & c
\end{pmatrix},
\qquad
D=D^T,\quad b\in\mathbb R^n,\quad \operatorname{tr}D+c=0,
\]
then
\[
\Theta\cdot C\Theta
=
\Theta'\cdot D\Theta'
+
2(b\cdot\Theta')\Theta_{n+1}
+
c\Theta_{n+1}^2 .
\]
Writing
\[
D=B+\frac{\operatorname{tr}D}{n}I_n,
\qquad B=B^T,\quad \operatorname{tr}B=0,
\]
and using \(c=-\operatorname{tr}D\) and
\(|\Theta'|^2=1-\Theta_{n+1}^2\) on \(\mathbb S^n\), the scalar trace part is
a multiple of
\( \Theta_{n+1}^2-\frac1{n+1}\).
Hence every \(v_2\in\mathcal H_2(\mathbb S^n)\) can be written uniquely as
\[
v_2(\Theta)
=
a\left(\Theta_{n+1}^2-\frac1{n+1}\right)
+
\sum_{i=1}^n b_i\Theta_i\Theta_{n+1}
+
\Theta'\cdot B\Theta',
\]
where
\( a,b_1,\dots,b_n\in\mathbb R\), \( B=B^T,\quad \operatorname{tr}B=0\).
The first term belongs to the radial physical sector, the second family belongs
to the first physical angular sector, and the last term belongs to the
degree-two physical angular sector.

Under the stereographic transfer \(P\), the last family gives precisely the
trace-free affine Jacobi fields
\[
x\cdot B\nabla U,
\qquad B=B^T,\quad \operatorname{tr}B=0,
\]
and these are excluded by the condition
\(\phi\perp_{\dot H^s}T_U\mathcal M_{\rm cone}\). Moreover, the degree-two
sector is strict on the orthogonal complement of these affine modes by
Corollary~\ref{cor:sharp-degree-two-sector-gap}. The higher odd sectors
\(\ell\ge3\) and the higher even sectors \(\ell\ge4\) are also strictly above
\(\gamma_s\), by the classical sector bound and
Lemma~\ref{lem:appendix-even-ge-four}.

Consequently, equality in \eqref{eq:sharp-affine-gap} holds if and only if
\(\phi=Pv_2\),
where
\[
v_2(\Theta)
=
a\left(\Theta_{n+1}^2-\frac1{n+1}\right)
+
\sum_{i=1}^n b_i\Theta_i\Theta_{n+1}.
\]
Conversely, every such \(Pv_2\) is orthogonal to
\(T_U\mathcal M_{\rm cone}\), lies in conformal degree \(2\), and therefore
satisfies
\[
\mathcal Q_U(Pv_2)
=
\gamma_s[Pv_2]_{\dot H^s(\mathbb R^n)}^2.
\]
Thus the equality statement follows. In particular, \(\Gamma_s\le\gamma_s\),
and hence \(\Gamma_s=\gamma_s\).
\end{proof}

\section{Global compactness}\label{sec:global}

In this section we carry out the global part of the proof of the stability
theorem. After the local coercive estimate, it remains to rule out loss of compactness along
sequences of almost extremizers. The argument follows the general
concentration--compactness and profile-decomposition strategy for critical
problems, in the spirit of the works of Lions~\cite{Lions85,Lions84,Lions84-1}
and Struwe~\cite{Struwe84}, and, in the fractional setting, of
G\'erard~\cite{Gerard98} and Palatucci--Pisante~\cite{PP14}.

\subsection{Affine normalization}

We next turn to the global part of the argument. The starting point is an
affine normalization adapted to the fractional polar projection body introduced
by Haddad--Ludwig~\cite{HL24}. Set
\[
p_u(\xi):=A_\xi(u)^{1/(2s)},
\qquad
K_u:=\{\xi\in\mathbb R^n:p_u(\xi)\le1\}.
\]
\begin{lemma}
\label{lem:affine-normalization}
Let \(0<s<1\), \(n>2s\), and let \(0\ne u\in\dot H^s(\mathbb R^n)\). Then there
exists \(M\in SL(n)\) such that, with \(Su:=T_{M,0}u\),
\[
\|Su\|_{L^{2_s^*}(\mathbb R^n)}
=
\|u\|_{L^{2_s^*}(\mathbb R^n)},
\qquad
\mathcal E_{\operatorname{aff},s}(Su)
=
\mathcal E_{\operatorname{aff},s}(u),
\]
and
\[
[Su]_{\dot H^s(\mathbb R^n)}^2
\le
C_{n,s}\mathcal E_{\operatorname{aff},s}(u)^2 .
\]
\end{lemma}

\begin{proof}
Since \(u\ne0\), one has \(A_\xi(u)>0\) for every
\(\xi\in\mathbb S^{n-1}\). Indeed, if \(A_\xi(u)=0\) for some \(\xi\), then the
Fourier representation of \(A_\xi\) forces \(\widehat u\) to be supported on
the hyperplane \(\xi^\perp\), hence \(\widehat u=0\) a.e.

We first record a uniform convexification property of \(K_u\). If \(2s\ge1\),
then \(p_u\) is a norm, by Minkowski's inequality, and \(K_u\) is an
origin-symmetric convex body. If \(0<2s<1\), then \(p_u^{2s}\) is subadditive;
hence, by Carath\'eodory's theorem~\cite[Theorem~1.1.4]{Schneider14},
\begin{equation}
\label{eq:comparability}
\operatorname{conv}(K_u)\subset C_{n,s}K_u .
\end{equation}
Thus, in all cases, there exists an origin-symmetric convex body comparable to
\(K_u\), with constants depending only on \(n\) and \(s\).

Applying John's ellipsoid theorem~\cite[Theorem~10.12.2]{Schneider14} to
\(\operatorname{conv}(K_u)\), and using \eqref{eq:comparability}, we obtain
an ellipsoid \(E_u\subset K_u\) with \(E_u=C_{n,s}^{-1}E\) such that
\( |E_u|\ge c_{n,s}|K_u|\).
Choose \(L\in SL(n)\) such that
\[
L(E_u)=B_R,
\qquad
R
=
\left(\frac{|E_u|}{|B_1|}\right)^{1/n}
\ge
c_{n,s}^{1/n}
\left(\frac{|K_u|}{|B_1|}\right)^{1/n}.
\]
Set \(M:=L^{-1}\) and \(Su:=T_{M,0}u\). Since \(M\in SL(n)\), both the critical
\(L^{2_s^*}\)-norm and the affine energy are invariant under \(T_{M,0}\). Also,
the transformation rule for \(K_u\) gives
\[
K_{Su}=M^{-1}K_u=L(K_u).
\]
Since \(E_u\subset K_u\), we have
\( B_R=L(E_u)\subset K_{Su}\).
Therefore, for every \(\xi\in\mathbb S^{n-1}\),
\[
p_{Su}(\xi)\le R^{-1},
\qquad
A_\xi(Su)\le R^{-2s}.
\]
Averaging and using \eqref{eq:intro-average-A},
we get
\[
[Su]_{\dot H^s(\mathbb R^n)}^2
\le
C_{n,s}R^{-2s}
\le
C_{n,s}
\left(\frac{|K_u|}{|B_1|}\right)^{-\frac{2s}{n}}
=
C_{n,s}\mathcal E_{\operatorname{aff},s}(u)^2 .
\]
This proves the estimate and the lemma.
\end{proof}

The purpose of Lemma~\ref{lem:affine-normalization} is to remove the
anisotropic escape allowed by the affine invariance. Although
\(\mathcal E_{\operatorname{aff},s}\) is invariant under the volume-preserving
affine action, the classical \(\dot H^s\)-seminorm need not be bounded along an
affine orbit. The lemma selects a representative for which the classical
fractional energy is controlled by the affine energy.

\begin{lemma}
\label{lem:normalized-compactness}
Let \(n\ge2\) and \(0<s<1\). Assume that
\(v_k\in\dot H^s(\mathbb R^n)\) satisfies
\( \|v_k\|_{L^{2_s^*}(\mathbb R^n)}=1\),
\(\delta_{\operatorname{aff},s}(v_k)\to0\).
Then there exist
\(c_k\neq0\), \( A_k\in GL^+(n)\), \(x_k\in\mathbb R^n\)
such that
\begin{equation}
\label{eq:normalized-compactness-conclusion}
c_k^{-1}T_{A_k,x_k}v_k
\to U
\qquad
\text{strongly in }\dot H^s(\mathbb R^n).
\end{equation}
\end{lemma}

\begin{proof}
Since
\( \|v_k\|_{L^{2_s^*}(\mathbb R^n)}=1\), \( \delta_{\operatorname{aff},s}(v_k)\to0\),
we have
\begin{equation}
\label{eq:compact-E-near-sharp}
\mathcal E_{\operatorname{aff},s}(v_k)^2\to S_{n,s}.
\end{equation}
By Lemma~\ref{lem:affine-normalization}, after applying suitable affine
transformations and relabeling, we may assume that
\[
\|v_k\|_{L^{2_s^*}(\mathbb R^n)}=1,
\qquad
\mathcal E_{\operatorname{aff},s}(v_k)^2\to S_{n,s},
\qquad
[v_k]_{\dot H^s(\mathbb R^n)}\le C .
\]

We apply the profile decomposition of Palatucci--Pisante
\cite[Theorem~1.4]{PP14} to the bounded sequence \((v_k)\subset\dot H^s(\mathbb R^n)\).
After passing to a subsequence, there exist an at most countable index set
\(I\), profiles \(V^j\in\dot H^s(\mathbb R^n)\), scales
\(\lambda_{j,k}>0\), centers \(y_{j,k}\in\mathbb R^n\), and dislocations
\[
(g_{j,k}V^j)(x)
:=
\lambda_{j,k}^{-\frac{n-2s}{2}}
V^j\left(\frac{x-y_{j,k}}{\lambda_{j,k}}\right),
\]
such that, for \(i\neq j\),
\[
\frac{\lambda_{i,k}}{\lambda_{j,k}}
+
\frac{\lambda_{j,k}}{\lambda_{i,k}}
+
\frac{|y_{i,k}-y_{j,k}|^2}{\lambda_{i,k}\lambda_{j,k}}
\to\infty .
\]
Moreover, for every finite set \(F\subset I\),
\[
v_k
=
\sum_{j\in F}g_{j,k}V^j
+
r_k^F ,
\]
and, by the  Br\'ezis--Lieb lemma \cite{BL85},
\begin{equation}
\label{eq:Lq-profile-decoupling-compactness}
\|v_k\|_{L^{2_s^*}(\mathbb R^n)}^{2_s^*}
=
\sum_{j\in F}\|V^j\|_{L^{2_s^*}(\mathbb R^n)}^{2_s^*}
+
\|r_k^F\|_{L^{2_s^*}(\mathbb R^n)}^{2_s^*}
+
o_k(1).
\end{equation}
Furthermore, for some exhaustion \(F_N\uparrow\mathbb N\),
\[
\lim_{N\to\infty}
\limsup_{k\to\infty}
\|r_k^{F_N}\|_{L^{2_s^*}(\mathbb R^n)}
=
0 .
\]

By the directional decoupling of the affine energies, for every finite
\(F\subset I\),
\begin{equation}
\label{eq:compact-directional-decoupling}
A_\xi(v_k)
=
\sum_{j\in F}A_\xi(V^j)+A_\xi(r_k^F)+o_k(1),
\qquad
\text{uniformly for }\xi\in\mathbb S^{n-1}.
\end{equation}

By \eqref{eq:intro-aff-frac-sob-p=2}, \eqref{eq:compact-E-near-sharp}, \eqref{eq:Lq-profile-decoupling-compactness}
together with \eqref{eq:compact-directional-decoupling}, we obtain, for every
finite \(F\subset I\),
\begin{equation}
\label{eq:subadd-frac}
\begin{aligned}
S_{n,s}
=
\lim_{k\to\infty}\mathcal E_{\operatorname{aff},s}(v_k)^2
&\ge
\sum_{j\in F}\mathcal E_{\operatorname{aff},s}(V^j)^2
+
\liminf_{k\to\infty}
\mathcal E_{\operatorname{aff},s}(r_k^F)^2
\\
&\ge
S_{n,s}
\sum_{j\in F}
\|V^j\|_{L^{2_s^*}(\mathbb R^n)}^2
+
S_{n,s}
\liminf_{k\to\infty}
\|r_k^F\|_{L^{2_s^*}(\mathbb R^n)}^2
\\
&=
S_{n,s}
\sum_{j\in F}
\|V^j\|_{L^{2_s^*}(\mathbb R^n)}^2
+
S_{n,s}
\left(
1-
\sum_{j\in F}
\|V^j\|_{L^{2_s^*}(\mathbb R^n)}^{2_s^*}
\right)^{\frac{2}{2_s^*}} .
\end{aligned}
\end{equation}
Letting \(F\uparrow I\) in \eqref{eq:subadd-frac}, and using
\eqref{eq:Lq-profile-decoupling-compactness}, and the fact that \(0<2/2_s^*<1\), we arrive at
\[
1
\ge
\sum_{j\in I}
\left(
\|V^j\|_{L^{2_s^*}(\mathbb R^n)}^{2_s^*}
\right)^{\frac{2}{2_s^*}}
\ge
\left(
\sum_{j\in I}
\|V^j\|_{L^{2_s^*}(\mathbb R^n)}^{2_s^*}
\right)^{\frac{2}{2_s^*}}
=
1 .
\]
Hence equality can occur only if exactly  one profile \(\|V^j\|_{L^{2_s^*}(\mathbb R^n)}^{2_s^*}\) is nonzero. Denote the corresponding profile by
\(V\), and let \(g_k\) be the associated dislocation.  Then
\[
\|V\|_{L^{2_s^*}(\mathbb R^n)}=1,
\qquad
\mathcal E_{\operatorname{aff},s}(V)^2=S_{n,s}.
\]
Thus \(V\) is an extremal for the sharp affine fractional Sobolev inequality \eqref{eq:intro-aff-frac-sob-p=2}.

By the equality classification in Haddad--Ludwig~\cite[Theorem~1]{HL24},
there exist \(c\neq0\), \(A\in GL^+(n)\), and \(x_0\in\mathbb R^n\) such that
\(  V=c\,T_{A,x_0}^{-1}U\).
Replacing \(v_k\) by
\(c^{-1}T_{A,x_0}g_k^{-1}v_k\)
and relabeling, we may assume that
\[
v_k=U+\rho_k,
\qquad
\rho_k\rightharpoonup0
\quad\text{weakly in }\dot H^s(\mathbb R^n),
\qquad
\|\rho_k\|_{L^{2_s^*}(\mathbb R^n)}\to0 .
\]
Moreover,
\begin{equation}
\label{eq:compact-energy-to-bubble}
\mathcal E_{\operatorname{aff},s}(U+\rho_k)^2
\to
\mathcal E_{\operatorname{aff},s}(U)^2 .
\end{equation}

It remains to prove that \([\rho_k]_{\dot H^s(\mathbb R^n)}\to0\).
Since \(\rho_k\rightharpoonup0\) in \(\dot H^s(\mathbb R^n)\), and since the set of
functionals
\( \left\{ h\mapsto\mathfrak a_\xi(U,h): \xi\in\mathbb S^{n-1} \right\} \)
is compact in \(\dot H^s(\mathbb R^n))'\),
we have
\( \sup_{\xi\in\mathbb S^{n-1}}
|\mathfrak a_\xi(U,\rho_k)|\to0\) .
Moreover,
\[
A_\xi(U+\rho_k)
=
A_0+A_\xi(\rho_k)+o_k(1) \quad \text{uniformly in } \quad \xi.
\]
Hence \eqref{eq:compact-energy-to-bubble} implies
\begin{equation}
\label{eq:compact-Phi-A0Bk}
\Phi(A_0+A_\xi(\rho_k))\to\Phi(A_0)=A_0 .
\end{equation}

We claim that
\(\left\langle A_\xi(\rho_k)\right\rangle_{\mathbb S^{n-1}}\to0\).
Suppose not. Since \((\rho_k)\) is bounded in \(\dot H^s(\mathbb R^n)\), after
passing to a subsequence we may assume that
\[
b_k:=
\left\langle A_\xi(\rho_k)\right\rangle_{\mathbb S^{n-1}}==
m_{n,s}
\int_{\mathbb S^{n-1}}
\int_0^\infty
r^{n-1+2s}
|\widehat{\rho_k}(r\theta)|^2\,dr\,d\sigma(\theta)
\to b>0 .
\]
For each \(k\), define a probability measure \(\mu_k\) on \(\mathbb S^{n-1}\) by
\[
d\mu_k(\theta)
:=
\frac{m_{n,s}}{b_k}
\left(
\int_0^\infty
r^{n-1+2s}
|\widehat{\rho_k}(r\theta)|^2\,dr
\right)d\sigma(\theta).
\]
After passing to a  subsequence,
\(\mu_k\rightharpoonup\mu\) weakly as probability measures on \(\mathbb S^{n-1}\).
Then
\[
A_\xi(\rho_k)
=
        \frac{\left\langle A_\xi(\rho_k)\right\rangle_{\mathbb S^{n-1}}}{m_{n,s}}
\int_{\mathbb S^{n-1}}
|\xi\cdot\theta|^{2s}\,d\mu_k(\theta),
\]
and we have the uniform convergence
\[
\frac{ A_\xi(\rho_k)}{\left\langle A_\xi(\rho_k)\right\rangle_{\mathbb S^{n-1}}}
=
        \frac{1}{m_{n,s}}
\int_{\mathbb S^{n-1}}
|\xi\cdot\theta|^{2s}\,d\mu_k(\theta)
\to
B(\xi)
:=
\frac1{m_{n,s}}
\int_{\mathbb S^{n-1}}
|\xi\cdot\theta|^{2s}\,d\mu(\theta).
\]
Moreover,
\( B\ge0\), \(\left\langle B\right\rangle_{\mathbb S^{n-1}}=1\).
Thus \(B\not\equiv0\), and \(A_0+bB>A_0\) on a set of positive measure. Hence
\[
\Phi(A_0+ A_\xi(\rho_k))\to  \Phi(A_0+bB)>A_0 .
\]
contradicting \eqref{eq:compact-Phi-A0Bk}. Therefore with \eqref{eq:intro-average-A},
\[
\left\langle  A_\xi(\rho_k)\right\rangle_{\mathbb S^{n-1}}
=
m_{n,s}[\rho_k]_{\dot H^s(\mathbb R^n)}^2 \to0.
\]
Undoing all normalizations, we obtain parameters
\(c_k\neq0\), \(A_k\in GL^+(n)\), and \(x_k\in\mathbb R^n\) such that
\[
c_k^{-1}T_{A_k,x_k}v_k
\to U
\qquad
\text{strongly in }\dot H^s(\mathbb R^n).
\]
This proves \eqref{eq:normalized-compactness-conclusion}.
\end{proof}

\section{Proof of the global stability theorem and of the strict upper bound}
\label{sec:proof}

In this section we complete the proof of the main global results. The lower
bound in Theorem~\ref{thm:two-sided-affine-stability} follows from the local
stability estimate and the global compactness argument, while the upper bound is
reduced to the classical fractional Sobolev stability theorem of
Chen--Frank--Weth~\cite{CFW13}. The proof of
Theorem~\ref{thm:affine-strict-upper-bound} follows the degree-two perturbation
argument of K\"onig~\cite{Konig23}.

\begin{lemma}
\label{lem:nearest-point-local-cone}
Assume \(n\ge2\) and \(0<s<1\). There exist
\(\varepsilon_0>0\), \(R_0>0\), and \(C>0\) such that the following holds.

If
\( [w-U]_{\dot H^s(\mathbb R^n)}\le \varepsilon_0\),
then there exists \(V\in\mathcal M_{\rm cone}\), with
\([V-U]_{\dot H^s(\mathbb R^n)}\le R_0\),
such that
\begin{equation}
\label{eq:local-nearest-point}
[w-V]_{\dot H^s(\mathbb R^n)}
=
\inf_{\substack{W\in\mathcal M_{\rm cone}\\
[W-U]_{\dot H^s(\mathbb R^n)}\le R_0}}
[w-W]_{\dot H^s(\mathbb R^n)} .
\end{equation}
Moreover, if \(r=w-V\), then
\begin{equation}
\label{eq:nearest-point-orthogonality}
r\perp_{\dot H^s(\mathbb R^n)}T_V\mathcal M_{\rm cone}.
\end{equation}
Finally,
\begin{equation}
\label{eq:local-distance-comparison}
D_{\operatorname{aff}}(w,\mathcal M_{\rm cone})
\le
C[w-V]_{\dot H^s(\mathbb R^n)} .
\end{equation}
\end{lemma}

\begin{proof}
Let
\[
\operatorname{Sym}_0(n)
:=
\{B\in\mathbb R^{n\times n}:B=B^T,\ \operatorname{tr}B=0\}.
\]
For
\(a=(\alpha,y,\lambda,B)
\in
\mathbb R\times\mathbb R^n\times\mathbb R\times\operatorname{Sym}_0(n)\)
near \(0\), set
\(
\Xi(a)
:=
(1+\alpha)\,T_{e^\lambda e^B,y}^{-1}U
\).
Since \(U\) is radial, rotational parameters are redundant. The map \(\Xi\) is
\(C^2\) as a map into \(\dot H^s(\mathbb R^n)\), \(\Xi(0)=U\), and
\[
\operatorname{Im}D\Xi(0)
=
T_U\mathcal M_{\rm cone}
=
\operatorname{span}
\left\{
U,\ Z_0,\ \partial_{x_1}U,\dots,\partial_{x_n}U,\
x\cdot B\nabla U
\right\}.
\]
After restricting to the above non-redundant parameters, \(D\Xi(0)\) is
injective. Hence, after shrinking the parameter neighborhood, \(\Xi\) is a
\(C^2\)-embedding and parametrizes the local piece of
\(\mathcal M_{\rm cone}\) near \(U\).

Thus, for \(R_0>0\) small, the set
\(\mathcal M_{\rm cone}\cap B_{\dot H^s}(U,R_0)\)
is a compact finite-dimensional \(C^2\)-manifold chart. Therefore, if
\([w-U]_{\dot H^s}\le\varepsilon_0\) with \(\varepsilon_0\) sufficiently small,
the function
\(a\mapsto [w-\Xi(a)]_{\dot H^s(\mathbb R^n)}^2\)
has an interior minimizer \(a_w\). Set \(V=\Xi(a_w)\). Then
\eqref{eq:local-nearest-point} follows.

Since \(a_w\) is an interior minimizer, the first variation vanishes:
\[
\left\langle
w-V,\,
D\Xi(a_w)\tau
\right\rangle_{\dot H^s}
=
0
\qquad
\text{for every parameter direction }\tau .
\]
Because \(\operatorname{Im}D\Xi(a_w)=T_V\mathcal M_{\rm cone}\), this gives
\eqref{eq:nearest-point-orthogonality}.
Finally, since \(V\in\mathcal M_{\rm cone}\), write
\( V=c\,T_{A,x_0}^{-1}U\)
with \(A\) close to \(I\). Then \(T_{A,x_0}V=cU\), and hence
\[
\begin{aligned}
D_{\operatorname{aff}}(w,\mathcal M_{\rm cone})
&\le
[T_{A,x_0}w-cU]_{\dot H^s(\mathbb R^n)}
=
[T_{A,x_0}(w-V)]_{\dot H^s(\mathbb R^n)}
\le
C[w-V]_{\dot H^s(\mathbb R^n)} ,
\end{aligned}
\]
because \(A\) stays in a fixed neighborhood of the identity. This proves
\eqref{eq:local-distance-comparison}.
\end{proof}

\subsection{Proof of Theorem~\ref{thm:two-sided-affine-stability}}

\begin{proof}[Proof of Theorem~\ref{thm:two-sided-affine-stability}]

We first prove the local estimate. There exist \(\varepsilon_0>0\) and
\(c_0>0\), depending only on \(n\) and \(s\), such that if, for some
\(c_*\ne0\), \(A_*\in GL^+(n)\), and \(x_*\in\mathbb R^n\),
\[
w:=c_*^{-1}T_{A_*,x_*}u
\quad\text{satisfies}\quad
[w-U]_{\dot H^s(\mathbb R^n)}\le \varepsilon_0,
\]
then
\begin{equation}
\label{eq:local-affine-frac-stability}
\delta_{\operatorname{aff},s}(u)
\ge
c_0D_{\operatorname{aff}}(u,\mathcal M_{\rm cone})^2 .
\end{equation}

Indeed, by affine invariance and homogeneity,
\[
\delta_{\operatorname{aff},s}(u)
=
c_*^2\delta_{\operatorname{aff},s}(w),
\qquad
D_{\operatorname{aff}}(u,\mathcal M_{\rm cone})
=
|c_*|D_{\operatorname{aff}}(w,\mathcal M_{\rm cone}).
\]
Thus it suffices to prove the estimate for \(w\).

Applying Lemma~\ref{lem:nearest-point-local-cone} to \(w\), we obtain
\(V\in\mathcal M_{\rm cone}\), close to \(U\), such that,
\begin{equation}
\label{eq:local-proof-orthogonality}
r\perp_{\dot H^s(\mathbb R^n)}T_V\mathcal M_{\rm cone},
\qquad  r=w-V
\end{equation}
and
\begin{equation}
\label{eq:local-proof-distance-comparison}
D_{\operatorname{aff}}(w,\mathcal M_{\rm cone})
\le
C[r]_{\dot H^s(\mathbb R^n)}.
\end{equation}

We first note that the spectral gap at \(U\) persists uniformly for extremals
\(V\in\mathcal M_{\rm cone}\) sufficiently close to \(U\). Indeed, otherwise
there would exist \(V_j\in\mathcal M_{\rm cone}\), \(V_j\to U\), and
\(h_j\in\dot H^s(\mathbb R^n)\) such that
\[
h_j\perp_{\dot H^s(\mathbb R^n)}T_{V_j}\mathcal M_{\rm cone},
\qquad
[h_j]_{\dot H^s(\mathbb R^n)}=1,
\qquad
\mathcal Q_{V_j}(h_j)\to0 .
\]
Since \(\mathcal M_{\rm cone}\) is a finite-dimensional \(C^2\)-manifold near
\(U\), its tangent spaces converge:
\( T_{V_j}\mathcal M_{\rm cone}\to T_U\mathcal M_{\rm cone}\).
Thus the \(\dot H^s\)-projection of \(h_j\) onto \(T_U\mathcal M_{\rm cone}\)
tends to zero. If \(P_U\) denotes this projection and
\(g_j=h_j-P_Uh_j\),
then
\[
g_j\perp_{\dot H^s(\mathbb R^n)}T_U\mathcal M_{\rm cone},
\qquad
[g_j]_{\dot H^s(\mathbb R^n)}\to1 .
\]
Since \(P_Uh_j\in\ker\mathcal Q_U\), Proposition~\ref{pro:sharp-local-affine-coeff}
gives
\[
\mathcal Q_U(h_j)
=
\mathcal Q_U(g_j)
\ge
\gamma_s[g_j]_{\dot H^s(\mathbb R^n)}^2
\ge
\frac{\gamma_s}{2}
\]
for \(j\) large. On the other hand, the second variation depends continuously
on the base point along the local cone, hence
\[
\mathcal Q_{V_j}(h_j)
=
\mathcal Q_U(h_j)+o(1),
\]
contradicting \(\mathcal Q_{V_j}(h_j)\to0\). Therefore there exists
\(\gamma_0>0\) such that, for every \(V\in\mathcal M_{\rm cone}\) sufficiently
close to \(U\),
\begin{equation}
\label{eq:uniform-gap-near-U}
\mathcal Q_V(h)
\ge
\gamma_0[h]_{\dot H^s(\mathbb R^n)}^2
\qquad
\forall\,h\perp_{\dot H^s(\mathbb R^n)}T_V\mathcal M_{\rm cone}.
\end{equation}

Now \(V\) is an affine extremal, by \eqref{eq:local-proof-orthogonality},
Taylor expansion at \(V\),
together with \eqref{eq:uniform-gap-near-U}, yields
\[
\delta_{\operatorname{aff},s}(V+r)
=
\mathcal Q_V(r)+o([r]_{\dot H^s(\mathbb R^n)}^2)
\ge
\frac{\gamma_0}{2}[r]_{\dot H^s(\mathbb R^n)}^2,
\]
after reducing the local smallness threshold if necessary.

Combining this with
\eqref{eq:local-proof-distance-comparison}, we obtain
\[
\delta_{\operatorname{aff},s}(w)
\ge
cD_{\operatorname{aff}}(w,\mathcal M_{\rm cone})^2.
\]
Rescaling back to \(u\) proves \eqref{eq:local-affine-frac-stability}.

We now prove the global lower bound for \(n\ge2\). Suppose, by contradiction,
that no constant \(c_{n,s}>0\) satisfies
\[
\delta_{\operatorname{aff},s}(u)
\ge
c_{n,s}D_{\operatorname{aff}}(u,\mathcal M_{\rm cone})^2
\qquad
\forall u\in\dot H^s(\mathbb R^n).
\]
Then there exists a sequence \(u_k\ne0\) such that
\begin{equation}
\label{eq:false-global-frac}
\delta_{\operatorname{aff},s}(u_k)
<
\frac1k
D_{\operatorname{aff}}(u_k,\mathcal M_{\rm cone})^2 .
\end{equation}
Since both sides are homogeneous of degree \(2\), we may normalize
\( \|u_k\|_{L^{2_s^*}(\mathbb R^n)}=1\).

Using Lemma~\ref{lem:affine-normalization}, after applying a volume-preserving
affine transformation and relabeling, we may additionally assume that
\[
[u_k]_{\dot H^s(\mathbb R^n)}\le C
\left(\mathcal E_{\operatorname{aff},s}(u_k)+1\right).
\]
In particular, the failure condition \eqref{eq:false-global-frac} forces
\( \delta_{\operatorname{aff},s}(u_k)\to0 \).
Therefore Lemma~\ref{lem:normalized-compactness} applies. After passing to a
subsequence, there exist
\( c_k\ne0\),  \(A_k\in GL^+(n)\),  \(x_k\in\mathbb R^n\)
such that
\(c_k^{-1}T_{A_k,x_k}u_k \to U\) trongly in \(\dot H^s(\mathbb R^n)\).
Thus
\(D_{\operatorname{aff}}(u_k,\mathcal M_{\rm cone})\to0\).
For \(k\) sufficiently large, the local estimate
\eqref{eq:local-affine-frac-stability} gives
\[
\delta_{\operatorname{aff},s}(u_k)
\ge
c_0D_{\operatorname{aff}}(u_k,\mathcal M_{\rm cone})^2,
\]
contradicting \eqref{eq:false-global-frac}. Hence the global lower bound holds:
\[
\delta_{\operatorname{aff},s}(u)
\ge
c_{n,s}D_{\operatorname{aff}}(u,\mathcal M_{\rm cone})^2 .
\]

It remains to prove the upper bound. By the definition of
\(\mathcal E_{\operatorname{aff},s}\), the negative-mean inequality, and
\eqref{eq:intro-affine-energy},
\[
\begin{aligned}
\mathcal E_{\operatorname{aff},s}(v)^2
&=
m_{n,s}^{-1}
\left\langle
A_\xi(v)^{-\frac n{2s}}
\right\rangle_{\mathbb S^{n-1}}^{-\frac{2s}{n}}
\le
m_{n,s}^{-1}
\left\langle
A_\xi(v)
\right\rangle_{\mathbb S^{n-1}}
=
[v]_{\dot H^s(\mathbb R^n)}^2 .
\end{aligned}
\]
Therefore
\begin{equation}
\label{eq:affine-deficit-below-classical}
\delta_{\operatorname{aff},s}(v)
\le
[v]_{\dot H^s(\mathbb R^n)}^2
-
S_{n,s}\|v\|_{L^{2_s^*}(\mathbb R^n)}^2
=
\delta_{\operatorname{frac},s}(v).
\end{equation}
By the sharp fractional Sobolev stability estimate of
Chen--Frank--Weth~\cite[Theorem~1]{CFW13} or \eqref{eq:intro-CFW-stability},
\begin{equation}
\label{eq:CFW-upper-used}
\delta_{\operatorname{frac},s}(v)
\le
d_{\operatorname{frac}}(v,\mathcal M)^2 .
\end{equation}
Now fix arbitrary \(A\in GL^+(n)\), \(x_0\in\mathbb R^n\), and \(c\in\mathbb R\),
and set \(v=T_{A,x_0}u\).
Since \(cU\in\mathcal M\), \eqref{eq:affine-deficit-below-classical} and
\eqref{eq:CFW-upper-used} give
\[
\delta_{\operatorname{aff},s}(u)
=
\delta_{\operatorname{aff},s}(v)
\le
[v-cU]_{\dot H^s(\mathbb R^n)}^2
=
[T_{A,x_0}u-cU]_{\dot H^s(\mathbb R^n)}^2 .
\]
Taking the infimum over \(c\in\mathbb R\), \(A\in GL^+(n)\), and
\(x_0\in\mathbb R^n\), we obtain
\[
\delta_{\operatorname{aff},s}(u)
\le
D_{\operatorname{aff}}(u,\mathcal M_{\rm cone})^2 .
\]

If \(n=1\), then \(\mathbb S^0=\{-1,1\}\) and
\( A_1(u)=A_{-1}(u)=[u]_{\dot H^s(\mathbb R)}^2\).
Thus
\( \mathcal E_{\operatorname{aff},s}(u)^2=[u]_{\dot H^s(\mathbb R)}^2\),
and \(\mathcal M_{\rm cone}\) reduces to \eqref{eq:intro-classical-manifold}.
Therefore  Theorem~\ref{thm:two-sided-affine-stability} in dimension \(n=1\), \(0<s<1/2\), is exactly the
Chen--Frank--Weth stability \cite{CFW13}.

Finally, we prove that the exponent \(2\) is sharp. Choose
\(0\ne\psi\in\dot H^s(\mathbb R^n)\), \(\psi\perp_{\dot H^s(\mathbb R^n)}T_U\mathcal M_{\rm cone}\).
By the second variation at \(U\),
\begin{equation}
\label{eq:sharpness-deficit-expansion}
\delta_{\operatorname{aff},s}(U+t\psi)
=
t^2\mathcal Q_U(\psi)+o(t^2).
\end{equation}
On the other hand, since \(\mathcal M_{\rm cone}\) is a \(C^2\)
finite-dimensional manifold near \(U\), its local chart gives
\[
D_{\operatorname{aff}}(U+t\psi,\mathcal M_{\rm cone})
=
|t|\operatorname{dist}_{\dot H^s}
\bigl(\psi,T_U\mathcal M_{\rm cone}\bigr)
+
o(|t|).
\]
Using the orthogonality of \(\psi\), this becomes
\begin{equation}
\label{eq:sharpness-distance-expansion}
D_{\operatorname{aff}}(U+t\psi,\mathcal M_{\rm cone})
=
|t|[\psi]_{\dot H^s(\mathbb R^n)}
+
o(|t|).
\end{equation}
Combining \eqref{eq:sharpness-deficit-expansion} and
\eqref{eq:sharpness-distance-expansion}, we obtain
\[
\delta_{\operatorname{aff},s}(U+t\psi)
\le
C_\psi
D_{\operatorname{aff}}(U+t\psi,\mathcal M_{\rm cone})^2
\]
for \(|t|\) small. Hence no estimate with power strictly smaller than \(2\) can
hold uniformly. This proves the sharpness of the exponent and completes the
proof.
\end{proof}

\subsection{Proof of Theorem~\ref{thm:affine-strict-upper-bound}}

We finally show that the optimal global quotient is strictly smaller than the
sharp local spectral value. The argument follows the same general strategy as in
K\"onig's study of the classical Bianchi--Egnell quotient~\cite{Konig23}.

\begin{proof}[Proof of Theorem~\ref{thm:affine-strict-upper-bound}]
Let
\[
\rho=Pv_2,
\qquad
v_2(\Theta):=\Theta_{n+1}^2-\frac1{n+1},
\qquad \Theta\in\mathbb S^n ,
\]
where \(P\) is the stereographic conformal transfer. Then \(v_2\) is a
spherical harmonic of degree \(2\) on \(\mathbb S^n\), and \(\rho\) is radial.
Therefore by Proposition~\ref{pro:sharp-local-affine-coeff},
\[
\mathcal Q_U(\rho)= \mathcal Q_{\operatorname{frac},s}(\rho)
=
\gamma_s[\rho]_{\dot H^s(\mathbb R^n)}^2,
\qquad
\gamma_s=\frac{2s}{\frac n2+s+1}.
\]

For \(\varepsilon>0\), set
\( u_\varepsilon:=U+\varepsilon\rho \).
Since \(u_\varepsilon\) is radial,
by the cubic expansion along this degree-two direction
\cite[Page~4]{Konig23}, there exists \(b_{n,s}>0\) such that
\begin{equation}
\label{eq:strict-upper-numerator}
\delta_{\operatorname{aff},s}(u_\varepsilon)
=
        \delta_{\operatorname{frac},s}(u_\varepsilon)
        =
\gamma_s\varepsilon^2[\rho]_{\dot H^s(\mathbb R^n)}^2
-
b_{n,s}\varepsilon^3
+
o(\varepsilon^3).
\end{equation}

It remains only to compare the affine distance with \(\varepsilon\). Since
\(\mathcal M_{\rm cone}\) is a \(C^2\) finite-dimensional manifold near \(U\)
and
\( \rho\perp_{\dot H^s(\mathbb R^n)}T_U\mathcal M_{\rm cone}\),
the standard squared-distance expansion to a \(C^2\) submanifold gives
\begin{equation}
\label{eq:strict-upper-distance}
D_{\operatorname{aff}}(U+\varepsilon\rho,\mathcal M_{\rm cone})^2
=
\varepsilon^2[\rho]_{\dot H^s(\mathbb R^n)}^2
+
O(\varepsilon^4).
\end{equation}

Combining \eqref{eq:strict-upper-numerator} and
\eqref{eq:strict-upper-distance}, we obtain
\[
C_{\operatorname{aff},s} <\frac{\delta_{\operatorname{aff},s}(u_\varepsilon)}
{D_{\operatorname{aff}}(u_\varepsilon,\mathcal M_{\rm cone})^2}
=
\gamma_s
-
\frac{b_{n,s}}{[\rho]_{\dot H^s(\mathbb R^n)}^2}\varepsilon
+
o(\varepsilon)
<
\gamma_s,
\]
for all sufficiently small \(\varepsilon>0\).
This proves the theorem.
\end{proof}

\section{Spectral decomposition of the affine fractional Hessian}
\label{sec:sector-decomposition-affine-frac}

The purpose of this section is to identify the nullspace of the affine
fractional Hessian \(Q_U\) at the radial bubble
\[
U(x)=(1+|x|^2)^{-\frac{n-2s}{2}},
\qquad 0<s<1,\qquad n\ge2.
\]
We shall prove that
\[
T_U\mathcal M_{\rm cone}
=
\ker Q_U
=
\operatorname{span}
\left\{
U,\ Z_0,\ \partial_{x_1}U,\dots,\partial_{x_n}U,\ x\cdot B\nabla U
\right\},
\]
where
\(B=B^T\), \(\operatorname{tr}B=0\), \(Z_0=\frac{n-2s}{2}U+x\cdot\nabla U\).

\subsection{Physical angular sectors for the classical fractional Hessian}

For
\(\phi\in \dot H^s(\mathbb R^n)\), write, for a.e. \(r>0\),
\[
\phi(r,\theta)
=
\sum_{\ell=0}^{\infty}
\sum_{m=1}^{d_\ell}
f_{\ell,m}(r)Y_{\ell,m}(\theta),
\qquad x=r\theta ,
\]
where
\[
-\Delta_{\mathbb S^{n-1}}Y_{\ell,m}
=
\lambda_\ell Y_{\ell,m},
\qquad
\lambda_\ell=\ell(\ell+n-2),
\]
and
\begin{equation}
\label{eq:normalization}
        \left\langle
Y_{\ell,m}Y_{\ell,m'}
\right\rangle_{\mathbb S^{n-1}}
=
\delta_{mm'}.
\end{equation}
We choose \(Y_{0,1}\equiv1\), and set
\[
\phi_\ell(r,\theta)
:=
\sum_{m=1}^{d_\ell}
f_{\ell,m}(r)Y_{\ell,m}(\theta).
\]

Since the \(\dot H^s\)-inner product is \(O(n)\)-invariant, different physical
angular sectors are orthogonal. Thus
\begin{equation}
\label{eq:frac-Hs-sector-splitting}
[\phi]_{\dot H^s(\mathbb R^n)}^2
=
\sum_{\ell=0}^{\infty}
[\phi_\ell]_{\dot H^s(\mathbb R^n)}^2
=
\sum_{\ell=0}^{\infty}
\sum_{m=1}^{d_\ell}
[f_{\ell,m}Y_{\ell,m}]_{\dot H^s(\mathbb R^n)}^2 .
\end{equation}

The local weighted terms are diagonal in the same decomposition. Since \(U\) is
radial, we have
\begin{equation}
\label{eq:frac-weighted-L2-splitting}
\int_{\mathbb R^n}U^{2_s^*-2}\phi^2\,dx
=
|\mathbb S^{n-1}|
\sum_{\ell=0}^{\infty}
\sum_{m=1}^{d_\ell}
\int_0^\infty
U(r)^{2_s^*-2}f_{\ell,m}(r)^2r^{n-1}\,dr .
\end{equation}
Moreover, since every nonconstant spherical harmonic has zero mean,
\begin{equation}
\label{eq:frac-rank-one-splitting}
\int_{\mathbb R^n}U^{2_s^*-1}\phi\,dx
=
|\mathbb S^{n-1}|
\int_0^\infty
U(r)^{2_s^*-1}f_{0,1}(r)r^{n-1}\,dr .
\end{equation}

Combining
\eqref{eq:def-hessian-frac},
\eqref{eq:frac-Hs-sector-splitting},
\eqref{eq:frac-weighted-L2-splitting}, and
\eqref{eq:frac-rank-one-splitting}, we obtain
\begin{equation}
\label{eq:frac-Q-sector-decomposition}
Q_{\operatorname{frac},s}(\phi)
=
Q_{\operatorname{frac},s}^{(0)}(f_{0,1})
+
\sum_{\ell=1}^{\infty}
\sum_{m=1}^{d_\ell}
Q_{\operatorname{frac},s}^{(\ell)}(f_{\ell,m}).
\end{equation}
Here, for \(\ell\ge1\),
\[
\begin{aligned}
Q_{\operatorname{frac},s}^{(\ell)}(f)
&:=
[fY_{\ell,m}]_{\dot H^s(\mathbb R^n)}^2
\\
&\quad
-
|\mathbb S^{n-1}|
S_{n,s}(2_s^*-1)
\|U\|_{L^{2_s^*}(\mathbb R^n)}^{2-2_s^*}
\int_0^\infty
U(r)^{2_s^*-2}f(r)^2r^{n-1}\,dr ,
\end{aligned}
\]
while in the radial sector
\[
\begin{aligned}
Q_{\operatorname{frac},s}^{(0)}(f)
&:=
[f]_{\dot H^s(\mathbb R^n)}^2
-
|\mathbb S^{n-1}|
S_{n,s}(2_s^*-1)
\|U\|_{L^{2_s^*}(\mathbb R^n)}^{2-2_s^*}
\int_0^\infty
U(r)^{2_s^*-2}f(r)^2r^{n-1}\,dr
\\
&\quad
+
S_{n,s}(2_s^*-2)
\|U\|_{L^{2_s^*}(\mathbb R^n)}^{2-2\cdot 2_s^*}
\left(
|\mathbb S^{n-1}|
\int_0^\infty
U(r)^{2_s^*-1}f(r)r^{n-1}\,dr
\right)^2 .
\end{aligned}
\]
For the whole \(\ell\)-block we write
\(Q_{\operatorname{frac},s}[\phi_\ell] = \sum_{m=1}^{d_\ell} Q_{\operatorname{frac},s}^{(\ell)}(f_{\ell,m})\).

\subsection{Classical nondegeneracy}

We shall use the nondegeneracy result of
Chen--Frank--Weth~\cite{CFW13}, which states that
\[
\ker Q_{\operatorname{frac},s}
=
\operatorname{span}
\left\{
U,\ Z_0,\ \partial_{x_1}U,\dots,\partial_{x_n}U
\right\}.
\]

Since \(U\) is radial, the spherical harmonic decomposition
\eqref{eq:frac-Q-sector-decomposition} is invariant under
\(Q_{\operatorname{frac},s}\). Moreover,
\(U\) and \(Z_0\) are radial, whereas
\(\partial_{x_j}U(x)=U'(r)\theta_j\), \(j=1,\dots,n\),
belong to the first spherical harmonic sector \(\mathcal H_1\).
Therefore
\begin{equation}
\label{eq:frac-radial-kernel}
Q_{\operatorname{frac},s}[\phi_0]=0
\quad\Longleftrightarrow\quad
\phi_0\in\operatorname{span}\{U,Z_0\},
\end{equation}
and
\begin{equation}
\label{eq:frac-first-kernel}
Q_{\operatorname{frac},s}[\phi_1]=0
\quad\Longleftrightarrow\quad
\phi_1\in
\operatorname{span}
\{\partial_{x_1}U,\dots,\partial_{x_n}U\}.
\end{equation}

For every \(\ell\ge2\), the kernel is trivial:
\begin{equation}
\label{eq:frac-higher-kernel}
Q_{\operatorname{frac},s}[\phi_\ell]=0
\quad\Longrightarrow\quad
\phi_\ell=0.
\end{equation}
Moreover, we have the following lower bounded estimate.


\begin{lemma}
\label{lem:sharp-classical-frac-sector-bound}
Let \(n\ge2\), \(0<s<1\), and set
\[
        \Lambda_k
=
\frac{\Gamma(k+\frac n2+s)}
        {\Gamma(k+\frac n2-s)},
\qquad k=0,1,2,\dots .
\]
Let \(Y_{\ell,m}\in\mathcal H_\ell(\mathbb S^{n-1})\), and let
\(\phi(x)=f(r)Y_{\ell,m}(\theta)\), \(x=r\theta\).
Then, for every \(\ell\ge2\),
\begin{equation}
\label{eq:appendix-frac-sector-lower}
\mathcal Q_{\operatorname{frac},s}^{(\ell)}(f)
=
Q_{\operatorname{frac},s}(\phi)
\ge
\left(1-\frac{\Lambda_1}{\Lambda_\ell}\right)
[fY_{\ell,m}]_{\dot H^s(\mathbb R^n)}^2 .
\end{equation}
Here \(\mathcal H_\ell(\mathbb S^n)\) denotes the space of degree-\(\ell\) spherical harmonics on \(\mathbb S^n\).
\end{lemma}

\begin{proof}
Let \(P\) denote the conformal transfer from \(\mathbb S^n\) to
\(\mathbb R^n\),
\begin{equation}
\label{eq:def-conformal-transfor}
        (P\eta)(x)
=
J_\pi(x)^{\frac{1}{2_s^*}}\eta(\pi(x)),
\qquad
        J_\pi(x)=\left(\frac{2}{1+|x|^2}\right)^n.
\end{equation}
Set \(\eta=P^{-1}\phi\).  If \(x=r\theta\), then
\[
\pi(x)
=
\left(
\frac{2x}{1+|x|^2},
\frac{1-|x|^2}{1+|x|^2}
\right)
=
(\sin\vartheta\,\theta,\cos\vartheta),
\qquad
r=\tan\frac{\vartheta}{2}.
\]
Since the conformal factor is radial, the angular sector is preserved:
\(\eta(\vartheta,\theta) =  g(\vartheta)Y_{\ell,m}(\theta)\)
for a suitable one-variable function \(g\).

We now record the consequence of the separated spherical-harmonic basis on
\(\mathbb S^n\). By \cite[Theorem~1.5.1]{DX13},  in coordinates
\(\zeta=(\sin\vartheta\,\theta,\cos\vartheta)\),  \(\theta\in\mathbb S^{n-1}\),
the degree \(k\) spherical harmonics on \(\mathbb S^n\) are spanned by functions of
the form
\[
(\sin\vartheta)^j
C_{k-j}^{\,j+(n-1)/2}(\cos\vartheta)
Y_{j,a}(\theta),
\qquad 0\le j\le k .
\]
Therefore every element of \(\mathcal H_k(\mathbb S^n)\) contains only angular
degrees \(j\le k\) in the \(\theta\)-variable.  Since
\(Y_{\ell,m}\) is orthogonal on \(\mathbb S^{n-1}\) to all angular degrees
\(j<\ell\), we obtain
\(\eta_k=0\) for every  \(k<\ell\),
where \(\eta_k\) is the orthogonal projection of \(\eta\) onto
\(\mathcal H_k(\mathbb S^n)\).

By the conformal spectral representation of the linearized fractional Sobolev
quadratic form \cite[Section~3]{CFW13}, one has
\begin{align*}
Q_{\operatorname{frac},s}(\phi)=  Q_{\operatorname{frac},s}(P\eta)
&=
\sum_{k=\ell}^{\infty}
(\Lambda_k-\Lambda_1)\|\eta_k\|_{L^2(\mathbb S^n)}^2 \\
&\ge
\left(1-\frac{\Lambda_1}{\Lambda_\ell}\right)
\sum_{k=\ell}^{\infty}
\Lambda_k\|\eta_k\|_{L^2(\mathbb S^n)}^2
=
\left(1-\frac{\Lambda_1}{\Lambda_\ell}\right)
[\phi]_{\dot H^s(\mathbb R^n)}^2 .
\end{align*}
This proves the claim.
\end{proof}

\subsection{Funk--Hecke analysis of the affine correction}

We now analyze the genuinely affine part of the Hessian. Since
\(\mathcal R_s\) is a positive multiple of
\(\operatorname{Var}_{\xi}(L_\xi(\phi))\), it suffices to diagonalize this
variance with respect to the physical angular decomposition.

Recall that \eqref{eq:def-Lxi},
\[
L_\xi(\phi)=
\int_{\mathbb R^n}
|\xi\cdot\zeta|^{2s}
\widehat U(\zeta)\overline{\widehat \phi(\zeta)}\,d\zeta.
\]
Since \(U\) is radial, its Fourier
transform is radial:
\(\widehat U(\zeta)=\widehat U_0(|\zeta|)\).
By the Bochner formula for the Fourier transform of spherical harmonics
\cite[Theorem~3.10]{SW71}, writing \(\zeta=\rho\theta\), we have
\[
\widehat\phi(\rho\theta)
=
\sum_{\ell=0}^{\infty}
\sum_{m=1}^{d_\ell}
\widehat f_{\ell,m}(\rho)Y_{\ell,m}(\theta),
\]
where the harmless phase factors are absorbed into the radial transforms
\(\widehat f_{\ell,m}\).

Using polar coordinates in the Fourier variable, we obtain
\begin{equation}
\label{eq:Lxi-frac-full-before-FH}
\begin{aligned}
L_\xi(\phi)
&=
\int_{\mathbb R^n}
|\xi\cdot\zeta|^{2s}
\widehat U(\zeta)\overline{\widehat\phi(\zeta)}\,d\zeta \\
&=
\int_0^\infty
\rho^{n-1+2s}\widehat U_0(\rho)
\int_{\mathbb S^{n-1}}
|\xi\cdot\theta|^{2s}
\overline{\widehat\phi(\rho\theta)}\,d\sigma(\theta)\,d\rho
\\
&=
\sum_{\ell=0}^{\infty}
\sum_{m=1}^{d_\ell}
\int_0^\infty
\rho^{n-1+2s}
\widehat U_0(\rho)
\overline{\widehat f_{\ell,m}(\rho)}
\left(
\int_{\mathbb S^{n-1}}
|\xi\cdot\theta|^{2s}Y_{\ell,m}(\theta)\,d\sigma(\theta)
\right)d\rho .
\end{aligned}
\end{equation}
Here the factor \(\rho^{n-1+2s}\) comes from
\(d\zeta=\rho^{n-1}d\rho\,d\sigma(\theta)\) and
\(|\xi\cdot(\rho\theta)|^{2s}=\rho^{2s}|\xi\cdot\theta|^{2s}\).

By the Funk--Hecke formula \cite[Theorem~1.2.9]{DX13}, there are coefficients
\(\beta_{\ell,s}\) such that
\begin{equation}
\label{eq:appendix-FH-identity}
\frac1{|\mathbb S^{n-1}|}
\int_{\mathbb S^{n-1}}
|\xi\cdot\theta|^{2s}Y_{\ell,m}(\theta)\,d\sigma(\theta)
=
\beta_{\ell,s}Y_{\ell,m}(\xi).
\end{equation}
Substituting \eqref{eq:appendix-FH-identity} into
\eqref{eq:Lxi-frac-full-before-FH} gives
\[
L_\xi(\phi)
=
\sum_{\ell=0}^{\infty}
\sum_{m=1}^{d_\ell}
\rho_{\ell,s}Y_{\ell,m}(\xi)I_\ell(f_{\ell,m}),
\]
where \( \rho_{\ell,s}=\frac{\beta_{\ell,s}}{\beta_{0,s}}\),
and
\begin{equation}
\label{eq:def-If}
I_\ell(f)
=
|\mathbb S^{n-1}|\beta_{0,s}
\int_0^\infty
\rho^{n-1+2s}
\widehat U_0(\rho)\overline{\widehat f_\ell(\rho)}\,d\rho .
\end{equation}

Since \(Y_{0,1}\equiv1\), the radial contribution \(\ell=0\) is independent of
\(\xi\) and therefore does not contribute to the variance. Using \eqref{eq:def-var}, \eqref{eq:normalization},
\begin{equation}
\label{eq:variance-Lxi-frac-sector-block}
\operatorname{Var}_{\xi\in\mathbb S^{n-1}}
\bigl(L_\xi(\phi)\bigr)
=
\sum_{\ell=1}^{\infty}
\sum_{m=1}^{d_\ell}
\rho_{\ell,s}^2 I_\ell(f_{\ell,m})^2 .
\end{equation}

It remains to compute the coefficients \(\rho_{\ell,s}\).

First assume \(n\ge3\), and set
\( \lambda=\frac{n-2}{2}\).
By the Funk--Hecke formula \cite[Theorem~1.2.9]{DX13},
\begin{equation}
\label{eq:appendix-beta-Gegenbauer}
\beta_{\ell,s}
=
\frac{|\mathbb S^{n-2}|}
{|\mathbb S^{n-1}|\,C_\ell^\lambda(1)}
\int_{-1}^{1}
|t|^{2s}C_\ell^\lambda(t)(1-t^2)^{\lambda-\frac12}\,dt .
\end{equation}
Since
\[
C_\ell^\lambda(-t)=(-1)^\ell C_\ell^\lambda(t)
\qquad\text{\cite[Section~10.9~(16)]{EMOT81}},
\]
the integral in \eqref{eq:appendix-beta-Gegenbauer} vanishes for every odd
\(\ell\). Hence
\(\rho_{\ell,s}=0\) for every odd \(\ell\), \(n\ge3\).

We next compute the even coefficients. Let \(\ell=2m\). By
\cite[Section~10.9~(21)]{EMOT81},
\[
\begin{aligned}
C_{2m}^{\lambda}(\sqrt u)
&=
(-1)^m\frac{(\lambda)_m}{m!}
{}_2F_1
\left(
-m,\ m+\lambda;\ \frac12;\ u
\right)  \\
&=
(-1)^m\frac{(\lambda)_m}{m!}
\sum_{k=0}^{m}
\frac{(-m)_k(m+\lambda)_k}{\left(\frac12\right)_k}
\frac{u^k}{k!}.
\end{aligned}
\]

Then
\begin{align*}
\mathcal J_{2m,s}
&:=
\int_{-1}^{1}
|t|^{2s}C_{2m}^{\lambda}(t)(1-t^2)^{\lambda-\frac12}\,dt \\
&=
\int_0^1
u^{s-\frac12}(1-u)^{\lambda-\frac12}
C_{2m}^{\lambda}(\sqrt u)\,du
\\
&=
(-1)^m
\frac{(\lambda)_m}{m!}
\int_0^1
u^{s-\frac12}(1-u)^{\lambda-\frac12}
{}_2F_1
\left(
-m,\ m+\lambda;\ \frac12;\ u
\right)\,du
\\
&=
(-1)^m
\frac{(\lambda)_m}{m!}
B\left(s+\frac12,\lambda+\frac12\right)
{}_3F_2
\left(
\begin{matrix}
-m,\ m+\lambda,\ s+\frac12\\
\frac12,\ \lambda+s+1
\end{matrix}
;1
\right)
\\
&=
(-1)^m
\frac{(\lambda)_m}{m!}
B\left(s+\frac12,\lambda+\frac12\right)
\frac{
\left(\lambda+\frac12\right)_m
(-s)_m
}{
\left(\frac12\right)_m
(\lambda+s+1)_m
} .
\end{align*}
In the last step we used Saalsch\"utz's summation formula
\cite[Theorem~2.2.6]{AAR99},
\[
\begin{aligned}
&{}_3F_2
\left(
\begin{matrix}
-m,\ m+\lambda,\ s+\frac12\\
\frac12,\ \lambda+s+1
\end{matrix}
;1
\right)
=
\sum_{k=0}^{m}
\frac{
(-m)_k(m+\lambda)_k\left(s+\frac12\right)_k
}{
\left(\frac12\right)_k(\lambda+s+1)_k
}
\frac1{k!}
=
\frac{
\left(\lambda+\frac12\right)_m
(-s)_m
}{
\left(\frac12\right)_m
(\lambda+s+1)_m
}.
\end{aligned}
\]
On the other hand, by \cite[Corollary~2.2.3]{AAR99},
\[
\mathcal J_{0,s}
=
B\left(s+\frac12,\lambda+\frac12\right),
\qquad
C_{2m}^{\lambda}(1)
=
\frac{(\frac{1}{2}-m-\lambda)_{m}}{\left(\frac{1}{2}\right)_m}
=
\frac{
(\lambda)_m\left(\lambda+\frac12\right)_m
}{
m!\left(\frac12\right)_m
}.
\]

Therefore,
\begin{equation}
\label{eq:rho2ms-nge3}
\rho_{2m,s}
=
\frac{\beta_{2m,s}}{\beta_{0,s}}
=
\frac{\mathcal J_{2m,s}}
{C_{2m}^{\lambda}(1)\mathcal J_{0,s}}
=
(-1)^m
\frac{(-s)_m}{(\lambda+s+1)_m}
=
(-1)^m
\frac{(-s)_m}{(\frac n2+s)_m}.
\end{equation}

It remains to discuss the case \(n=2\). With the normalized real Fourier basis
\[
Y_0=1,\qquad
Y_{\ell,c}(\vartheta)=\sqrt2\cos(\ell\vartheta),
\qquad
Y_{\ell,s}(\vartheta)=\sqrt2\sin(\ell\vartheta),
\quad \ell\ge1,
\]
the Funk--Hecke coefficient is
\[
\beta_{\ell,s}
=
\frac1{2\pi}
\int_0^{2\pi}
|\cos\vartheta|^{2s}\cos(\ell\vartheta)\,d\vartheta .
\]
Since \(|\cos\vartheta|^{2s}\) is even and \(\pi\)-periodic, this coefficient
vanishes for every odd \(\ell\). For \(\ell=2m\), the classical Fourier
coefficient formula for \(|\cos\vartheta|^{2s}\) gives
\begin{equation}
\label{eq:rho2ms-n2}
\rho_{2m,s}
=
\frac{
\displaystyle
\int_0^{2\pi}
|\cos\vartheta|^{2s}\cos(2m\vartheta)\,d\vartheta
}{
\displaystyle
\int_0^{2\pi}
|\cos\vartheta|^{2s}\,d\vartheta
}
=
\frac{\Gamma(s+1)^2}
{\Gamma(s-m+1)\Gamma(s+m+1)}
=
(-1)^m
\frac{(-s)_m}{(s+1)_m}.
\end{equation}
The formulas
\eqref{eq:rho2ms-nge3} and \eqref{eq:rho2ms-n2} combine to give, for all
\(n\ge2\),
\begin{equation}
\label{eq:rho2ms}
\rho_{2m,s}
=
\frac{\beta_{2m,s}}{\beta_{0,s}}
=
(-1)^m
\frac{(-s)_m}{(\frac n2+s)_m},
\qquad m\ge0.
\end{equation}
In particular,
\begin{equation}
\label{eq:appendix-rho-ratio}
\frac{|\rho_{2m+2,s}|}{|\rho_{2m,s}|}
=
\frac{m-s}{m+\frac n2+s}
<1,
\qquad m\ge1 .
\end{equation}
Together with the odd vanishing above, we have
\begin{equation}
\label{eq:appendix-rho-odd-zero}
\rho_{\ell,s}=0
\qquad
\text{for every odd }\ell .
\end{equation}

For \(\ell\ge1\) and \(m=1,\dots,d_\ell\), set
\[
\psi_{\ell,m}(r,\theta)
:=
f_{\ell,m}(r)Y_{\ell,m}(\theta).
\]
Since the affine correction is rotation invariant, the corresponding
one-component form depends on the angular degree \(\ell\), but not on the index
\(m\). We therefore define
\[
\mathcal R_s^{(\ell)}(f_{\ell,m})
:=
\mathcal R_s(\psi_{\ell,m}).
\]

\begin{lemma}
\label{lem:appendix-affine-correction-bound}
Assume \(0<s<1\), \(n\ge2\), and let \(\ell\ge1\). Then, for every
\(m=1,\dots,d_\ell\),
\begin{equation}
\label{eq:appendix-affine-correction-bound}
\mathcal R_s^{(\ell)}(f_{\ell,m})
\le
\frac{n+2s}{s}
\rho_{\ell,s}^2
[\psi_{\ell,m}]_{\dot H^s(\mathbb R^n)}^2 .
\end{equation}
In particular,
\(\mathcal R_s^{(\ell)}\equiv0\) for every odd \(\ell\).
\end{lemma}

\begin{proof}
Fix one component
\[
\psi_{\ell,m}(r,\theta)
=
f_{\ell,m}(r)Y_{\ell,m}(\theta).
\]
By \eqref{eq:def-Rs}, \eqref{eq:variance-Lxi-frac-sector-block},
\begin{equation}
\label{eq:R-sector-Iell}
\mathcal R_s^{(\ell)}(f_{\ell,m})
=
m_{n,s}^{-1}
\frac{n+2s}{sA_0}
        \operatorname{Var}_{\xi\in\mathbb S^{n-1}}
\bigl(L_\xi(\psi_{\ell,m})\bigr)=
m_{n,s}^{-1}
\frac{n+2s}{sA_0}
\rho_{\ell,s}^2
|I_\ell(f_{\ell,m})|^2 .
\end{equation}

It remains to estimate \(I_\ell(f_{\ell,m})\). By \eqref{eq:intro-Axi-frac}, \eqref{eq:intro-average-A},
the Bochner decomposition, and \eqref{eq:appendix-FH-identity},
\begin{equation}
\label{eq:Axfy}
\begin{aligned}
m_{n,s}[\psi_{\ell,m}]_{\dot H^s(\mathbb R^n)}^2
&=
\left\langle
A_\xi(\psi_{\ell,m})
\right\rangle_{\mathbb S^{n-1}}
=
|\mathbb S^{n-1}|\beta_{0,s}
\int_0^\infty
\rho^{n-1+2s}
|\widehat f_{\ell,m}(\rho)|^2\,d\rho .
\end{aligned}
\end{equation}
Here \(\beta_{0,s}=m_{n,s}\).

By \eqref{eq:Axu-1}, \eqref{eq:def-If}, \eqref{eq:Axfy}, and Cauchy--Schwarz,
\begin{equation}
\label{eq:Ifs}
\begin{aligned}
|I_\ell(f_{\ell,m})|^2
&\le
\left(
|\mathbb S^{n-1}|\beta_{0,s}
\int_0^\infty
\rho^{n-1+2s}|\widehat U_0(\rho)|^2\,d\rho
\right)\\
&\quad\times
\left(
|\mathbb S^{n-1}|\beta_{0,s}
\int_0^\infty
\rho^{n-1+2s}|\widehat f_{\ell,m}(\rho)|^2\,d\rho
\right)=
A_0\,m_{n,s}
[\psi_{\ell,m}]_{\dot H^s(\mathbb R^n)}^2 .
\end{aligned}
\end{equation}
Substituting \eqref{eq:Ifs} into \eqref{eq:R-sector-Iell} gives
\[
\mathcal R_s^{(\ell)}(f_{\ell,m})
\le
\frac{n+2s}{s}
\rho_{\ell,s}^2
[\psi_{\ell,m}]_{\dot H^s(\mathbb R^n)}^2 .
\]
This proves \eqref{eq:appendix-affine-correction-bound}. Finally,
\eqref{eq:appendix-rho-odd-zero} gives
\(\rho_{\ell,s}=0\) for every odd \(\ell\),
and hence \(\mathcal R_s^{(\ell)}\equiv0\) for every odd \(\ell\).
\end{proof}

\subsection{Identification of the degree-two affine kernel}

The degree-two sector is the only physical angular sector in which new affine
Jacobi fields may occur. We now identify this kernel.

\begin{lemma}
\label{lem:degree-two-affine-sector}
Assume \(0<s<1\) and \(n\ge2\).
\[
\left\{\phi_2:\mathcal Q_U[\phi_2]=0\right\}
=
\left\{
x\cdot B\nabla U:
B=B^T,\ \operatorname{tr}B=0
\right\}.
\]
\end{lemma}

\begin{proof}
Set
\(  h(r)=rU'(r)\).
We first note that every trace-free symmetric affine direction is a zero mode.
Indeed, if \(B=B^T\) and \(\operatorname{tr}B=0\), then
\(t\mapsto T_{e^{tB},0}^{-1}U\) is a curve of affine extremals. Differentiating
at \(t=0\) gives \(x\cdot B\nabla U\in\ker\mathcal Q_U\).
Moreover,
\[
x\cdot B\nabla U
=
rU'(r)\,\theta\cdot B\theta,
\qquad x=r\theta,
\]
and \(\theta\cdot B\theta\in\mathcal H_2(\mathbb S^{n-1})\) whenever
\(\operatorname{tr}B=0\). Thus these zero modes lie in the degree-two sector.

It remains to show that there are no further zero modes. For a single
degree-two component \(f(r)Y_{2,m}(\theta)\), define
\[
\mathcal Q_U^{(2)}(f)
:=
Q_{\operatorname{frac},s}^{(2)}(f)
-
\mathcal R_s^{(2)}(f).
\]
By the definition of \(\mathcal R_s^{(2)}\),
\[
\mathcal Q_U^{(2)}(f)
=
Q_{\operatorname{frac},s}^{(2)}(f)
-
\alpha_2|I_2(f)|^2,
\qquad
\alpha_2
:=
m_{n,s}^{-1}
\frac{n+2s}{sA_0}\rho_{2,s}^2 .
\]

By \eqref{eq:frac-higher-kernel},
\(Q_{\operatorname{frac},s}^{(2)}\) is positive definite. We write
\[
\|f\|_2^2:=Q_{\operatorname{frac},s}^{(2)}(f),
\qquad
\langle f,g\rangle_2:=Q_{\operatorname{frac},s}^{(2)}(f,g).
\]

By \eqref{eq:appendix-frac-sector-lower}; \eqref{eq:Ifs}, with \(\ell=2\), we have
\[
|I_2(f)|^2
\le
\frac{A_0m_{n,s}}
{1-\frac{\Lambda_1}{\Lambda_2}}
Q_{\operatorname{frac},s}^{(2)}(f)
=
C_{n,s,U}\|f\|_2^2 .
\]
Therefore \(f\mapsto I_2(f)\) is continuous with respect to the Hilbert norm
\(\|f\|_2^2=Q_{\operatorname{frac},s}^{(2)}(f)\).
Hence, by the Riesz representation theorem, there exists
\(w\) such that
\( I_2(f)=\langle f,w\rangle_2 \).
Consequently,
\begin{equation}
\label{eq:degree-two-rank-one-form}
0\le \mathcal Q_U^{(2)}(f)
=
\|f\|_2^2
-
\alpha_2|\langle f,w\rangle_2|^2 .
\end{equation}
Therefore
\begin{equation}
\label{eq:alpha-w-bound}
\alpha_2\|w\|_2^2\le1 .
\end{equation}

For each basis element \(Y_{2,m}\), choose a trace-free symmetric matrix
\(B_m\) such that
\( Y_{2,m}(\theta)=\theta\cdot B_m\theta\).
Then
\[
x\cdot B_m\nabla U
=
rU'(r)Y_{2,m}(\theta)
=
h(r)Y_{2,m}(\theta)
\]
is a zero mode. Hence
\( \mathcal Q_U^{(2)}(h)=0\).
Using \eqref{eq:degree-two-rank-one-form} and \eqref{eq:alpha-w-bound}, we get
\[
0
=
\|h\|_2^2
-
\alpha_2|\langle h,w\rangle_2|^2
\ge
\|h\|_2^2
\left(1-\alpha_2\|w\|_2^2\right)
\ge0.
\]
Thus equality holds in Cauchy--Schwarz. Hence \(h\) is proportional to \(w\),
and
\(\alpha_2\|w\|_2^2=1\).
Now suppose \(\mathcal Q_U^{(2)}(f)=0\). Then
\[
\|f\|_2^2
=
\alpha_2|\langle f,w\rangle_2|^2
\le
\alpha_2\|f\|_2^2\|w\|_2^2
=
\|f\|_2^2 .
\]
Again equality holds in Cauchy--Schwarz, and therefore \(f\) is proportional to
\(w\), equivalently to \(h=rU'\). Hence
\[
\ker \mathcal Q_U^{(2)}
=
\operatorname{span}\{rU'\}
\]
for each fixed degree-two angular component.

Therefore
\[
\mathcal Q_U[\phi_2]=0
\quad\Longleftrightarrow\quad
f_{2,m}\in\operatorname{span}\{rU'\}
\quad\text{for every }m=1,\dots,d_2.
\]
Thus  there exists \(B=B^T\), \(\operatorname{tr}B=0\),
such that
\[
\phi_2(r,\theta)
=
rU'(r)\sum_{m=1}^{d_2}c_mY_{2,m}(\theta)
        =
rU'(r)\,\theta\cdot B\theta
=
x\cdot B\nabla U(x).
\]
This proves the lemma.
\end{proof}

In fact, the proof of Lemma~\ref{lem:degree-two-affine-sector} gives, for a fixed
degree-two spherical harmonic \(Y_{2,m}\), one has
\begin{equation}
\label{eq:degree-two-projection-form}
\mathcal Q_U^{(2)}(f)
=
\|f\|_2^2
-
\frac{
|\langle f,h\rangle_2|^2
}{
\|h\|_2^2
},
\qquad
h(r)=rU'(r),
\end{equation}

As a consequence, after projecting out the affine degree-two zero modes, the
degree-two sector satisfies the same sharp lower bound as the classical
fractional Hessian.

\begin{corollary}
\label{cor:sharp-degree-two-sector-gap}
Let \(0<s<1\) and \(n\ge2\). Let
\[
\phi_2
\perp_{\dot H^s(\mathbb R^n)}
\left\{
x\cdot B\nabla U:
B=B^T,\ \operatorname{tr}B=0
\right\}.
\]
Then
\begin{equation}
\label{eq:sharp-degree-two-sector-gap}
\mathcal Q_U[\phi_2]
\ge
\gamma_s[\phi_2]_{\dot H^s(\mathbb R^n)}^2,
\qquad
\gamma_s
=
\frac{2s}{\frac n2+s+1}.
\end{equation}
\end{corollary}

\begin{proof}
By orthogonality of degree-two spherical harmonics, it is enough to work
componentwise. Fix \(m\), and consider the component
\(f_m(r)Y_{2,m}(\theta)\). Choose \(B_m=B_m^T\), \(\operatorname{tr}B_m=0\),
such that
\[
Y_{2,m}(\theta)=\theta\cdot B_m\theta .
\]
Then the corresponding affine zero mode is
\[
x\cdot B_m\nabla U(x)
=
h(r)Y_{2,m}(\theta),
\qquad
h(r)=rU'(r).
\]
The assumption
\[
\phi_2
\perp_{\dot H^s}
\left\{
x\cdot B\nabla U:
B=B^T,\ \operatorname{tr}B=0
\right\}
\]
therefore implies, for every \(m=1,\dots,d_2\),
\begin{equation}
\label{eq:component-Hs-orthogonal-h}
f_m(r)Y_{2,m}(\theta)
\perp_{\dot H^s(\mathbb R^n)}
h(r)Y_{2,m}(\theta).
\end{equation}

By Lemma~\ref{lem:sharp-classical-frac-sector-bound},
\[
Q_{\operatorname{frac},s}^{(2)}(f)
\ge
\gamma_s[fY_{2,m}]_{\dot H^s(\mathbb R^n)}^2,
\qquad
\gamma_s
=
1-\frac{\Lambda_1}{\Lambda_2}
=
\frac{2s}{\frac n2+s+1}.
\]
Define the nonnegative bilinear form
\[
b(f,g)
:=
\langle f,g\rangle_2
-
\gamma_s
\langle fY_{2,m},gY_{2,m}\rangle_{\dot H^s(\mathbb R^n)} .
\]

Using \eqref{eq:component-Hs-orthogonal-h}, we have
\(\langle f_m,h\rangle_2 = b(f_m,h)\).
Therefore, by \eqref{eq:degree-two-projection-form}, with Cauchy's inequality,
\[
\begin{aligned}
\mathcal Q_U^{(2)}(f_m)
-
\gamma_s[f_mY_{2,m}]_{\dot H^s(\mathbb R^n)}^2
&=
b(f_m,f_m)
-
\frac{|b(f_m,h)|^2}{\|h\|_2^2}\\
&\ge
        \frac{|b(f_m,h)|^2}{\|h\|_2^2}-\frac{|b(f_m,h)|^2}{\|h\|_2^2}=0.
\end{aligned}
\]
Summing over \(m=1,\dots,d_2\) yields
\[
\mathcal Q_U[\phi_2]
=
\sum_{m=1}^{d_2}\mathcal Q_U^{(2)}(f_m)
\ge
\gamma_s
\sum_{m=1}^{d_2}
[f_mY_{2,m}]_{\dot H^s(\mathbb R^n)}^2
=
\gamma_s[\phi_2]_{\dot H^s(\mathbb R^n)}^2.
\]
This proves \eqref{eq:sharp-degree-two-sector-gap}.
\end{proof}

The next lemma rules out any additional loss of coercivity in the higher even
sectors. Together with the odd-sector vanishing of the affine correction, this
completes the positivity analysis outside the symmetry and degree-two affine
modes.

\begin{lemma}
\label{lem:appendix-even-ge-four}
Let \(0<s<1\), \(n\ge2\), and let \(\ell\ge4\) be even. Then
\[
\mathcal Q_U[\phi_\ell]
\ge
\gamma_s[\phi_\ell]_{\dot H^s(\mathbb R^n)}^2,
\qquad
\gamma_s
=
\frac{2s}{\frac n2+s+1}.
\]
\end{lemma}

\begin{proof}
Let \(\ell=2m\ge4\). By \eqref{eq:rho2ms} and  \eqref{eq:appendix-rho-ratio},
\[
|\rho_{\ell,s}|
\le
|\rho_{4,s}| =
\frac{s(1-s)}
{\left(\frac n2+s\right)\left(\frac n2+s+1\right)}
\qquad
\text{for every even }\ell\ge4 .
\]
Hence Lemma~\ref{lem:appendix-affine-correction-bound} gives
\begin{equation}
\label{eq:correction-estimate}
\begin{aligned}
\mathcal R_s^{(\ell)}(f_{\ell,m})
&\le
\frac{n+2s}{s}
|\rho_{4,s}|^2
[f_{\ell,m}Y_{\ell,m}]_{\dot H^s(\mathbb R^n)}^2
\\
&=
\frac{2s(1-s)^2}
{\left(\frac n2+s\right)\left(\frac n2+s+1\right)^2}
[f_{\ell,m}Y_{\ell,m}]_{\dot H^s(\mathbb R^n)}^2 .
\end{aligned}
\end{equation}

On the other hand, since \(\ell\ge4\), by Lemma~\ref{lem:sharp-classical-frac-sector-bound}, and
\eqref{eq:correction-estimate}
\[
\begin{aligned}
\mathcal Q_U^{(\ell)}(f_{\ell,m})&=
Q_{\operatorname{frac},s}^{(\ell)}(f_{\ell,m})- \mathcal R_s^{(\ell)}(f_{\ell,m})
\ge
\eta_{4,s}
[f_{\ell,m}Y_{\ell,m}]_{\dot H^s(\mathbb R^n)}^2  .
\end{aligned}
\]
where
\[
\eta_{4,s}
:=
1-\frac{\Lambda_1}{\Lambda_4}
-
\frac{2s(1-s)^2}
{\left(\frac n2+s\right)\left(\frac n2+s+1\right)^2}.
\]
A direct computation gives
\[
\begin{aligned}
\eta_{4,s}-\gamma_s
&=
2s
\left[
\frac{\left(\frac n2+1-s\right)(n+5)}
{\left(\frac n2+1+s\right)
        \left(\frac n2+2+s\right)
        \left(\frac n2+3+s\right)}
-
\frac{(1-s)^2}
{\left(\frac n2+s\right)
        \left(\frac n2+1+s\right)^2}
\right]
>0 .
\end{aligned}
\]

Therefore
\[
\mathcal Q_U[\phi_\ell]
=
\sum_{m=1}^{d_\ell}\mathcal Q_U^{(\ell)}(f_m)
\ge
\gamma_s[\phi_\ell]_{\dot H^s(\mathbb R^n)}^2
\]
in the even sector \(\ell\ge4\).
\end{proof}

\subsection{Conclusion: kernel identification}

We now assemble the sectorial information obtained above.  By the sector decompositions of
\(Q_{\operatorname{frac},s}\) and \(\mathcal R_s\),
\[
\mathcal Q_U(\phi)
=
\mathcal Q_U[\phi_0]
+
\sum_{\ell=1}^{\infty}\mathcal Q_U[\phi_\ell],
\]
where
\[
\mathcal Q_U[\phi_\ell]
=
Q_{\operatorname{frac},s}[\phi_\ell]
-
\mathcal R_s[\phi_\ell].
\]

In the radial sector, \(L_\xi(\phi_0)\) is independent of \(\xi\). Hence
\[
\mathcal R_s[\phi_0]=0,
\qquad
\mathcal Q_U[\phi_0]
=
Q_{\operatorname{frac},s}[\phi_0].
\]
Therefore, by \eqref{eq:frac-radial-kernel}
\[
\mathcal Q_U[\phi_0]=0
\quad\Longleftrightarrow\quad
\phi_0\in\operatorname{span}\{U,Z_0\}.
\]

In the first angular sector, \(\rho_{1,s}=0\) by
\eqref{eq:appendix-rho-odd-zero}. Hence
\[
\mathcal R_s[\phi_1]=0,
\qquad
\mathcal Q_U[\phi_1]
=
Q_{\operatorname{frac},s}[\phi_1].
\]
Thus, by \eqref{eq:frac-first-kernel}
\[
\mathcal Q_U[\phi_1]=0
\quad\Longleftrightarrow\quad
\phi_1\in
\operatorname{span}
\{\partial_{x_1}U,\dots,\partial_{x_n}U\}.
\]

In the degree-two sector, Lemma~\ref{lem:degree-two-affine-sector} gives
\[
\left\{\phi_2:\mathcal Q_U[\phi_2]=0\right\}
=
\left\{
x\cdot B\nabla U:
B=B^T,\ \operatorname{tr}B=0
\right\}.
\]

It remains to consider the higher sectors. If \(\ell\ge3\) is odd, then
\(\rho_{\ell,s}=0\), and hence
\[
\mathcal Q_U[\phi_\ell]
=
Q_{\operatorname{frac},s}[\phi_\ell]=0.
\]
If \(\ell\ge4\) is even, Lemma~\ref{lem:appendix-even-ge-four} gives the strict
lower bound
\[
\mathcal Q_U[\phi_\ell]
>
\gamma_s[\phi_\ell]_{\dot H^s(\mathbb R^n)}^2
\qquad
\text{for every nonzero }\phi_\ell .
\]
Hence no nonzero kernel element can occur in any even sector \(\ell\ge4\).

Combining the radial, first, second, and higher-sector conclusions, we obtain
\[
\ker\mathcal Q_U
=
\operatorname{span}
\left\{
U,\ Z_0,\ \partial_{x_1}U,\dots,\partial_{x_n}U,\
x\cdot B\nabla U
\right\}.
\]

\noindent\textbf{Conflict of interest:}
Authors state no conflict of interest.

\noindent\textbf{Data Availability Statement:}
Data sharing is not applicable to this article as no datasets were generated or analysed during the current study.

\noindent{\bf Acknowledgement.}
G-D.~Li was supported by NSFC (No.12561019).


\providecommand{\href}[2]{#2}
\providecommand{\arxiv}[1]{\href{http://arxiv.org/abs/#1}{arXiv:#1}}
\providecommand{\url}[1]{\texttt{#1}}
\providecommand{\urlprefix}{DOI }

\end{document}